\documentclass[11pt, a4paper]{article}

\usepackage{amsmath,amssymb, amsthm}
\usepackage{mathtools}
\usepackage[utf8]{inputenc}
\usepackage{mathrsfs}

\usepackage{fullpage}

\usepackage{authblk}

\usepackage{dsfont}
\usepackage{bbm}

\usepackage{cancel}

\usepackage{graphicx}
\graphicspath{ {images/} }

\usepackage{wrapfig}

\usepackage{verbatim}

\usepackage[normalem]{ulem}

\usepackage[dvipsnames]{xcolor}

\usepackage{hyperref}

\hypersetup{
	colorlinks   = true, 
	urlcolor     = {blue!90!black}, 
	linkcolor    = {blue!90!black}, 
	citecolor   = {red!90!black} 
}

\allowdisplaybreaks

\newcommand{\R}{\mathbb{R}}
\newcommand{\N}{\mathbb{N}}
\newcommand{\Z}{\mathbb{Z}}

\renewcommand{\P}{\mathbb{P}}
\newcommand{\E}{\mathbb{E}}
\newcommand{\PP}[1]{\mathbb{P}\left\{#1\right\}}

\newcommand{\1}{\mathds{1}}
\newcommand{\A}{\mathcal{A} }
\renewcommand{\AA}{\mathscr{A} }

\newcommand{\Viktor}[1]{\textcolor{black!88!black}{#1}}
\definecolor{indigo}{rgb}{0.29, 0.0, 0.51}

\newcommand{\Peter}[1]{\textcolor{black!92!black}{#1}}

\newcommand{\LN}{\mathbb N}

\newcommand{\fP}{\mathbb P}

\newcommand{\LZ}{\mathbb Z}

\newcommand{\set}[1]{{\left\{ #1 \right\}}}

\renewcommand{\P}{\mathbb{P}}

\AtBeginDocument{

}

\theoremstyle{plain}
\newtheorem{thm}{Theorem}[section]
\newtheorem{prop}[thm]{Proposition}
\newtheorem{lem}[thm]{Lemma}
\newtheorem{cor}[thm]{Corollary}
\newtheorem{rmk}[thm]{Remark}

\newtheorem{Example}[thm]{Example}

\theoremstyle{definition}
\newtheorem{con}[thm]{Condition}
\newtheorem{defi}[thm]{Definition}

\author[1,2]{Viktor Bezborodov \thanks{Email: \texttt{viktor.bezborodov@uni-goettingen.de}}} 
\author[3]{
	Luca Di Persio \thanks{Email: \texttt{luca.dipersio@univr.it}}}
\author[4]{Peter Kuchling \thanks{Email: \texttt{peter.kuchling@hsbi.de}} }

\affil[1]{\emph{The University of Goettingen, Institute for Mathematical Stochastics, Germany} 
}
\affil[2]{
	{\emph{Wroc\l{}aw University of Science and Technology, Faculty of Electronics, Poland }
}}
\affil[3]{
	{\emph{The University of Verona, Department of Computer Science, Italy}}}

\affil[4]{\emph{Hochschule Bielefeld - University of Applied Sciences and Arts, Faculty of Engineering and Mathematics, Germany} 
}

\title{Explosion and non-explosion 
	for the continuous-time frog model}

\begin{document}
	
	\maketitle
	
	\begin{abstract}
		We consider the continuous-time frog model on $\Z$. At time $t = 0$, there are  $\eta (x)$  particles at $x\in \Z$, each of which is represented by a random variable. In particular,
		\Viktor{ \Peter{$(\eta(x))_{x \in \Z }$} is a collection  of independent 
			random variables
			with a common  distribution $\mu$, $\mu(\Z_+) = 1$, $\Z_+ := \N \cup \{0\}$,
			$\N = \{1,2,3,...\}$.}
		The particles at the origin are {\it active}, all other ones being assumed as {\it dormant}, or {\it sleeping}.
		Active particles perform a simple
		symmetric continuous-time random walk
		in $\Z $
		(that is, a random walk with $\exp(1)$-distributed jump times and jumps $-1$ and $1$, each with probability $1/2$),
		independently of all other particles.
		Sleeping particles stay still until the first
		arrival of an active particle to their location;
		upon arrival they become active and start their own simple random walks.
		Different sets of
		conditions are given ensuring explosion, respectively non-explosion, of the continuous-time frog model.
		\Viktor{Our results show \Peter{in particular} that if $\mu$ is the distribution of $e^{Y \ln Y}$ with a non-negative  random variable $Y$ satisfying $\E Y < \infty$, then a.s.\ no explosion occurs. On the other hand, if $a \in (0,1)$ and $\mu$ is the distribution 
			of $e^X$, where $\P \{X \geq t \} = t^{-a}$, $t \geq 1$,
			then explosion occurs a.s.}
		The proof relies on a certain type 
		of comparison to a percolation model
		which we call totally asymmetric discrete inhomogeneous
		Boolean percolation.

	\end{abstract}


	\section{Introduction}

	
	Denote by $\A_t$  the set of sites visited by 
	\Viktor{active particles} by the time $t$.
	In this paper we 
	investigate the various conditions on $\mu$
	ensuring that the system explodes,  respectively does not explode,
	in  finite time. 
	We exclude a trivial case
	and  assume throughout that $\mu(0) < 1$.
	If $\eta(0) = 0$,
	then we add one active particle 
	at the origin at time $t = 0$.
	
	\begin{defi}\label{def linear spread}
		We say that the system explodes (in finite time) if there exists $t \in (0,\infty)$ such that $\A_t $ is infinite.
	\end{defi}

	Our aim is to analyse and give conditions for explosion and non-explosion of the continuous-time frog model.
	An equivalent definition 
	of explosion \Viktor{on $\Z$} is 
	the following:
	\label{die Huelle = shell}
	there are no sleeping particles left in a finite time.
	This equivalence is not entirely trivial \Viktor{and may not be true for the frog model on a general graph;} 
	\Viktor{in the one dimensional case} it follows 
	from the arguments on Page \pageref{page: explosion occurs in both directions} at the beginning of Section \ref{sec non expl}.
	
	The behaviour of the frog dynamics can be distinguished as follows:
	\begin{center}
		linear spread - superlinear spread but no explosion - explosion in  finite time
	\end{center}
	We will review these concepts in more detail in Section \ref{sec proof outline}.

	The question of explosion
	is a classical question 
	in the theory of branching processes
	\cite[Chapter 5, Section 9]{HarrisBook}
	and  is an important consideration
	in a general construction 
	of an interacting particle system 
	\cite{EW03}.
	An explosion
	is a phenomenon known
	to take place
	in  first-passage percolation models
	if a node can have sufficiently many neighbors
	\cite{ExplFpp}.
	A different type of explosion is considered
	in \cite{CD16}, where conditions
	for accumulation of an unbounded number
	of particles within a compact set
	\Peter{are} given for a branching random
	walk with non-negative displacements.
	In \cite{PP94} necessary and sufficient 
	conditions for explosion 
	are given for first passage percolation 
	with unit exponential weights 
	on a spherically symmetric rooted tree. 
	For a tree in which 
	every vertex in generation $n$
	has $f(n+1)$ children
	the probability of an explosion 
	is shown to be $1$ if and only if $\sum_{n=1} ^\infty
	\frac{1}{f(n)} < \infty$; the probability is $0$ otherwise.
	More general (non-exponential) weights are also briefly discussed
	in \cite{PP94}. 
	Under broad assumptions,
	conditions for 
	explosion of first passage percolation on 
	spherically symmetric trees with arbitrary weight distribution are obtained in 
	\cite{FFP_trees_expl}. As one might expect,
	a lot depends on the interplay between 
	$f$
	and the behavior of the weight distribution function 
	near zero.
	
	An explosion can occur 
	for certain classes of stochastic differential equations.
	It is sometimes referred to as a blow-up.
	Conditions for explosion and non-explosion constitute 
	a part of the classical theory
	\cite[Chapter VI, Section 4]{IkedaWat}.
	A drift condition 
	ensuring explosion 
	for a multidimensional 
	equation is given in \cite{Chow13Explosion}.
	Various terms may cause
	explosion 
	in a stochastic differential equation 
	with jumps \cite{BY16}.

	The frog model was introduced 
	in \cite{shapeFrog}
	where a shape theorem for the model was proven
	for the frog model in discrete time with $\mu = \delta _1$, i.e., at $t=0$ there is one frog at each site.
	The asymptotic properties
	of the spread 
	have been studied 
	for the frog model on various graphs:
	on the integer lattice
	\cite{shapeFrogRandom}, trees \cite{FrogsOnTrees},
	Cayley graphs \cite{FrogsOnCayleyGraphs},
	as well as multitype model on 
	the integer lattice \cite{DHL19}.
	A shape theorem in every dimension
	and finer results in dimension one
	for the continuous-time  model
	have been obtained in \cite{stocCombust, CQR07, CQR09}.
	A possibility of explosion 
	for the continuous-time frog model
	with a general (not necessarily exponential) distribution
	of the time
	between jumps
	was demonstrated in \cite{frog1}.

	The results of this paper can be framed 
	in terms of
	the cover time, that is, 
	the time when every site of a graph is visited
	by an active particle. 
	Explosion means that the cover time 
	of $\Z$ 
	is finite for the continuous-time frog model;
	if no explosion occurs a.s., then the cover time 
	is infinite. 
	For the discrete-time 
	model the asymptotics of cover 
	time have been studied 
	on various finite graphs,
	in particular  trees \cite{FrogsOnTrees?, HJJ19}
	and  tori and sequences of expander graphs \cite{BFHM20}.
	
	In this paper we give sufficient conditions for explosion and non-explosion of the continuous-time 
	frog model. \Peter{In the proofs we rely on a comparison with a certain kind of auxiliary percolation model.} Using a similar proof technique, in \cite{frogL}
	the linear and superlinear spread 
	of the continuous-time frog model was studied.
	Further description of the technique can be found
	in Section~\ref{sec proof outline}.
	\Viktor{The results of this paper demonstrate 
		flexibility and versatility of the technique.
		We expect it to be applicable in various other settings when addressing the questions such as spread rate or explosion for stochastic growth models.}

	The paper is organized as follows. 
	In Section~\ref{sec results} we formulate and discuss the main results. In Section \ref{sec tadibp} an auxiliary percolation model is introduced. In Section \ref{sec proof outline} further discussion and the main ideas of the proof are collected. Sections \ref{sec non expl}
	and \ref{sec explosion proof} contain the proofs of non-explosion and explosion, respectively.

	\section{Main results and discussion} \label{sec results}

	In this section, we give sufficient conditions on the initial distribution $\mu$ of sleeping particles which lead to explosion or non-explosion.
	Let $A\colon\N \to(0,\infty)$
	be a non-decreasing function 
	which we interpret as a varying speed 
	for the continuous-time frog model. 
	We remark here that the word `speed' is used loosely in this paper. We mostly have in mind an average speed over a certain interval, that is, the ratio of the  distance covered to the  time elapsed 
	since the start of movement of one or several particles.

	For $i,j \in \N$,  set 
	$\AA(i): = \sum\limits_{z = 1} ^i \frac{1}{A(z)}$,
	$\AA(i,i+j): = \AA(i+j) - \AA(i) = \sum\limits_{z = i+1} ^{i+j} \frac{1}{A(z)}$, and $\AA(0)=0$. Furthermore,
	let $a_0 = 0$ and for $i\in\N$ 
	set $a_i := \frac{i!}{\left(\AA(i)\right) ^i}$.
	
	\begin{thm}\label{main_theorem}
		\begin{enumerate}
			\item[(i)] 
			Assume that
			\begin{equation}\label{A_series_diverges}
				\sum_{{z=1}}^\infty\frac{1}{A(z)}=\infty
			\end{equation}
			and
			\begin{equation}\label{skeevy = unpleasant, creepy}
				\sum_{i=0}^\infty\mu\left(\left[a_i,\infty\right)\right)<\infty.
			\end{equation}
			Then almost surely no explosion occurs.
			\item[(ii)] Assume that
			\begin{equation}\label{A_series_converges}
				\sum_{{z=1}}^\infty\frac{1}{A(z)}<\infty
			\end{equation}
			and there exists $\rho>1$ such that
			\begin{equation}\label{poprawa = improvement}
				\sum_{m=1}^\infty\prod_{i=1}^m\mu([0,A(m)^{{\rho}i}])<\infty.
			\end{equation}
			Then an explosion occurs almost surely.

		\end{enumerate}
	\end{thm}

	\begin{rmk}
		If $A$ is bounded and \eqref{skeevy = unpleasant, creepy} holds, then 
		by \cite[Theorem 1.2 (i)]{frogL}
		not only a.s.\ no explosion occurs,
		but we  know even that the spread is a.s.\ linear. 
		On the other hand, condition \eqref{poprawa = improvement}
		resembles the conditions in \cite[Theorem 1.2 (ii)]{frogL}.
	\end{rmk}
	
	Condition \eqref{skeevy = unpleasant, creepy} 
	is  shown in Section \ref{sec non expl} to imply that in a certain sense `many'  sites $z \in \N$
	are traveled over at  speed below $A(z)$. Together with \eqref{A_series_diverges} this is then shown to imply  the absence of an explosion a.s. 
	On the other hand, \eqref{poprawa = improvement}
	is used in Section \ref{sec explosion proof} to 
	obtain that in some sense `most' sites $z \in \N $
	are traveled over at speed exceeding $A(z)$. 
	This together with \eqref{A_series_converges} 
	is then shown to imply a.s.\ explosion. 
	A deeper discussion of the proof ideas can be found in Section~\ref{sec proof outline}.
	Note that since $\mu$ is concentrated 
	on $\Z_+$, the function
	$\rho \mapsto \sum_{m=1}^\infty\prod_{i=1}^m\mu([0,A(m)^{\rho i}])$
	is non-decreasing,
	therefore condition
	\eqref{poprawa = improvement}
	is stronger than 
	\begin{equation*}
		\sum_{m=1}^\infty\prod_{i=1}^m\mu([0,A(m)^{i}])<\infty.
	\end{equation*}

	For non-explosion we need to control the tails of the initial condition \eqref{skeevy = unpleasant, creepy} so that there are not 
	too many dormant frogs at the beginning.
	On the other hand, in the condition for explosion we 
	require
	of the initial distribution 
	to be sufficiently heavy \eqref{poprawa = improvement}. 
	Taking $A(x) = \frac{1}{\ln (x+1) - \ln x }$
	in Theorem \ref{main_theorem}, (i), we get 
	\begin{cor}
		Assume that 
		\begin{displaymath}
			\sum_{i=1}^\infty\mu\left(\left[\frac{i!}{(\ln(i+1))^i},\infty\right)\right)<\infty,
		\end{displaymath}
		or equivalently 
		\begin{displaymath}
			\sum_{i=2}^\infty\mu\left(\left[\frac{i!}{(\ln i)^i},\infty\right)\right)<\infty.
		\end{displaymath}
		Then a.s.\ no explosion occurs.
	\end{cor}
	
	Applying the inequality $1 - a \leq e^{-a}$, $a \geq 0$,
	to the left hand side of \eqref{poprawa = improvement}, we get
	\begin{cor}\label{perusal =  read and examine}
		Assume that \eqref{A_series_converges} 
		holds and for some
		$\rho>1$ 
		\begin{equation}\label{straight-laced = strict morals}
			\sum_{m=1}^\infty
			\exp\Big\{- \sum_{i=1}^m \mu((A(m)^{{\rho}i}, \infty)) \Big\}
			<\infty.
		\end{equation}
		Then an explosion occurs almost surely.
	\end{cor}
	
	In the proof we link the frog model
	to an asymmetric inhomogeneous percolation.
	This is described in more detail in Sections 
	\ref{sec tadibp} and
	\ref{sec proof outline}.

	\begin{rmk}
		\Viktor{
			In this paper we focus on the one-dimensional
			process. 
			A coupling 
			argument 
			(\cite[Lemma 4.1]{stocCombust} or  \cite[Proposition 2.3]{frogL})
			implies  
			that whenever explosion occurs on $\Z$ with probability one
			it also occurs   on $\Z ^d$, $d \in \N$, with probability one.
			We expect furthermore that explosion on $\Z$ implies  explosion on regular trees by a similar coupling argument.
			Whether explosion on $\Z$ implies explosion for every connected graph with finite degrees
			containing a copy of the integer line as a subgraph
			is not clear.
	}\end{rmk}
	
	\begin{rmk}
		Denote by $\mathscr{F}_n$ the $\sigma$-algebra generated
		by the walks of the particles started 
		from within the set $\{-n,-n+1, ..., n\}$, that is
		\begin{equation}\label{sich in die Sonne liegen}
			\mathscr{F}_n = \sigma \{S _{t} ^{(x ,j)}, \eta(x), t\geq 0, -n\leq x\leq n,
			1 \leq j \leq \eta (x) \}. 
		\end{equation}
		The  event $\{\text{explosion occurs}\}$
		is in $\sigma \big(\bigcup_{n \in \N} \mathscr{F}_n \big)$
		and yet  independent of $\mathscr{F}_n$ for every $n \in \N$. 
		Therefore by the $0$-$1$ law, $\P \{\text{explosion occurs}\} \in \{0,1 \}$ for any distribution $\mu$.
	\end{rmk}
	We close the discussion of our result by providing 
	\Viktor{explicit examples of $\mu$ leading to explosion or non-explosion}.
	\begin{Example}
		Assume that for some $a \in (0,1)$ for large $b>0$ 
		\begin{equation}\label{ersichtlich = evident, apparent}
			\mu( (b, \infty)) \geq (\ln b)^{-a}. 
		\end{equation}
		Then by taking $A(n) = n^{\alpha}$ with $\alpha > 1$ we find 
		for $m \in \N$
		\begin{multline*}
			\sum_{i=1}^m \mu((A(m)^{{\rho}i}, \infty)) 
			= 
			\sum_{i=1}^m \mu((m^{ \alpha \rho i}, \infty)) 
			\geq  \sum_{i=1}^m \frac{1}{(\alpha \rho i \ln m )^{a}}
			\\
			= \frac{1}{(\alpha \rho  \ln m )^{a}}
			\sum_{i=1}^m \frac{1}{i^a}
			\geq \frac{1}{(\alpha \rho  \ln m )^{a}} \frac{c m ^{1-a}}{1-a}
		\end{multline*}
		for some $c> 0$. Hence 
		\begin{equation*}
			\sum_{m=1}^\infty
			\exp\Big\{- \sum_{i=1}^m \mu((A(m)^{{\rho}i}, \infty)) \Big\}
			\leq \sum_{m=1}^\infty
			\exp\Big\{- 
			\frac{c m ^{1-a}}{(1-a)(\alpha \rho  \ln m )^{a}}
			\Big\} < \infty.
		\end{equation*}
		By Corollary \ref{perusal =  read and examine} an explosion occurs a.s.
		\Viktor{It follows  that if $e^X \sim \mu$ ($\mu$ is the distribution 
			of $e^X$), where $X$ is a Pareto distribution with 
			$$
			\P \{ X \geq t \} \leq \frac{1}{t^a}, 
			\ \ \ t \geq 1
			$$
			for $a \in (0,1)$,
			then explosion occurs a.s.
		}
		Note that condition \eqref{ersichtlich = evident, apparent} is much weaker than 
		the explosion condition in \cite[Remark 2.5]{frog1} \Peter{, which was given by
			\begin{displaymath}
				\mu([2^{4n^2+1}n^{8n^2+1},\infty))\geq\frac{1}{n-1}\text{ for }n\geq 2.
			\end{displaymath}
		}
	\end{Example}
	\begin{Example}
		\Viktor{    Assume that  $e^{Y \ln Y} \sim \mu$, 
			where a non-negative random variable $Y$ has a finite expectation.
			Take $A(x) = \frac{1}{\ln (x+1) - \ln x }$. We write for large $i$ 
			$$a_i = \frac{i!}{(\ln(i+1))^i} \geq e^{0.95i \ln i}$$  and 
			\begin{equation*}
				\mu \big( [a_i, \infty ) \big) = \P \{ e^{Y \ln Y} \geq  a_i \} 
				\leq \P \{   e^{Y \ln Y}  \geq  e^{ 0.95 i \ln i} \}
				\leq  \P \{   Y  \geq  0.9 i  \}.
			\end{equation*}
			Hence 
			\eqref{skeevy = unpleasant, creepy} holds,
			and Theorem \ref{main_theorem}, (i), implies that no explosion occurs a.s.}
	\end{Example}
	\begin{rmk}
		\Viktor{ Let us place two preceding examples in the context of \cite{frogL}: we know that 
			if $\E Y < \infty$ for a non-negative random variable $Y$
			and $e^{Y} \sim \mu$,
			then the spread is linear in time \cite[Theorem 1.2, (i)]{frogL}. On the other hand, 
			\cite[Theorem 1.2, (iii)]{frogL} implies
			that the spread is superlinear if
			$\E Y = \infty$ and
			$e^{Y \ln ^2 Y} \sim \mu$. }
	\end{rmk}

	\section{Totally asymmetric discrete inhomogeneous Boolean percolation (TADIBP)}
	\label{sec tadibp}
	
	In this section we introduce the percolation process which is used to analyze the explosion of the frog model.  We introduce a general TADIBP model which is used in Section \ref{sec proof outline} to define a percolation process corresponding to the frog model.
	Let $\{\psi _z\}_{z \in \Z}$ be a collection of
	independent 
	$\Z _+$-valued random  variables
	with distributions  $p_k ^{(z)} = \PP{\psi _z = k}$, $z \in \Z$.
	We consider a germ-grain model
	with germs at the sites of  $\Z$
	and grains of the form $[x, x+\psi _x]$.
	The distribution of a $\Z_+$-valued random variable
	$\psi _x$ depends on the location $x$,
	hence the model is inhomogeneous in space. 
	Germ-grain models are well 
	known and typically treated in homogeneous settings
	\cite[Section 6.5]{MeckeBook13}.
	The spatially homogeneous version of the model 
	we present below was introduced by Lamperti \cite{Lamp70}
	and was later considered  
	in \cite{KW06} and \cite{Zer18}. 
	A continuous-space version of the model 
	is treated in \cite{TABP}.
	We follow the 
	interpretation introduced in \cite{Lamp70}:
	At each site $x$, there is a fountain
	that wets integer sites in the interval
	$(x, x+\psi _x]$.

	We say that $x , y \in \Z$, $x \leq y$, are directly connected  (denoted by $x \xrightharpoondown{\Z}  y$)
	if there exists $z \leq x$, $z \in \Z$, such that $z + \psi _z \geq y$.
	We say that $x$  and $y$ are connected (denoted by $x \xrightarrow{\Z} y$)
	if they are directly connected, 
	or
	if there exists $z_1 \leq ... \leq z_n \in \Z$,
	$z_1 \leq x$, $z_n \leq y$, such that $x \in [z_1, z_1 + \psi _{z_1}]$, 
	$y \in [z_n, z_n + \psi _{z_n}]$,
	and $z_{j+1} \in [z_j, z_j + \psi _{z_j}]$ for $j = 1,2, ..., n-1$,
	or, equivalently, 
	$$x\xrightharpoondown{\Z}z_2\xrightharpoondown{\Z}\dotsb\xrightharpoondown{\Z}z_n\xrightharpoondown{\Z}y.$$
	For a subset $Q \subset \Z$,
	$x \xrightharpoondown{Q}  y$
	and 
	$x \xrightarrow{Q} y$ are
	defined in the same way with an additional requirement 
	that $x, y, z,  z_1, ..., z_n \in Q$ (in this paper we only consider 
	$Q = \Z$ and $Q = \Z _+$).
	We say that $x \in \Z$
	is \emph{wet} if
	the interval $[x-1,x]$
	is contained in $[y, y + \psi _y]$
	for some $y \in \Z$.
	In other words, 
	$x \in \Z$
	is {wet} if for some $y\in \Z $,  $y < x$ and $y + \psi _y \geq x$.
	The sites that are  not wet are said to be dry. 
	Note that
	$x$ is wet if and only if $x-1$ and $x$ are connected.
	We call the resulting random structure
	totally asymmetric discrete inhomogeneous Boolean
	percolation (TADIBP).
	When considering TADIBP on $\Z _+$,
	we also talk about `wet' sites, 
	with the understanding that both $x$ and $y$
	are required to be from $\Z_+$.
	Also,  
	we consider the origin
	to be wet for  TADIBP on $\Z_+$.

	\begin{defi}\label{Aufwand = Effort}
		For $m\in\LZ_+$, denote by $Y_m$ the difference between the rightmost site directly connected to $m$ and $m$, i.e.
		\begin{displaymath}
			Y_m=\max\set{l\colon m\xrightharpoondown{\LZ_+} l}-m.
		\end{displaymath}
	\end{defi}
	By definition, $m\xrightharpoondown{\LZ_+}m$ and hence, $Y_m\geq 0$. Also, by construction, $Y_0=\psi_0$ and for $m\in\LN$,
	\begin{displaymath}
		Y_m=\psi_m\lor(\psi_{m-1}-1)\lor\dots\lor(\psi_1-m+1)\lor(\psi_0-m).
	\end{displaymath}
	We say that $x$ is connected to infinity, denoted by  
	$x\xrightarrow{\Z_+} \infty $,
	if $x\xrightarrow{\Z_+} y $
	for every $y > x$.
	Note that for $x \in \Z_+$, 
	$x\xrightarrow{\Z_+} \infty $ if and only 
	if $Y _m > 0$ for all $m \geq x$.
	\begin{defi}
		We say that a system $\{\psi_x\}$ of random variables of the TADIBP percolates if there exists $x_0\in\LZ_+$ such that $x_0\xrightarrow{\Z_+}\infty$.
	\end{defi}
	

	
	\Viktor{The following lemma is \cite[Lemma 3.8]{frogL} with a typo corrected. For completeness we also give the proof.
		\begin{lem}\label{perc_sequence}
			Let $x \in \N$. A.s.\ on $\{ x \xrightarrow{ \Z _+ } \infty \}$,
			every site $y > x$ is wet, and there exists
			a (random) sequence $ x_0 < x_1 < x_2 < \dots $, $x_i \in \Z_+$,
			such that $x_0 \leq x < x_1$ and for every $i \in \Z _+$
			\begin{equation}\label{slink}
				x _{i + 1} \leq x _ i + \psi _{x _ i} <  x _{i + 2}.
			\end{equation}
			In particular, every $z \geq x$  belongs to no more than two
			intervals of the type $[x_i, x_i + \psi _{x _ i} ]$, $i \in \Z _+$.
		\end{lem}
		\begin{proof}
			By definition of $\xrightarrow{ \Z _+ }$, 
			every site $y > x$ is wet a.s.\ on $\{ x \xrightarrow{ \Z _+ } \infty \}$.
			Define the elements of the sequence $\{ x_i\}_{i \in \Z_+}$
			consecutively by setting 
			$$
			x_0 = \max\big\{y \in [0, x_0 ] \cap \N : y + \psi _y = \max\{z + \psi _z: z = x_0, x_0-1, ...,0 \}   \big\}
			$$
			and letting for $i \in \Z_+$
			\begin{equation}
				x_{i+1} = \max\big\{y \in [x_i + 1, x_i + \psi _{x _ i} ] \cap \N : y + \psi _y = \max\{z + \psi _z: z = x_i + 1, \dots, x _i + \psi _{x _i} \}   \big\}.
			\end{equation}
			In other words, $x_{i+1} \in [x_i + 1, x_i + \psi _{x _ i} ]$ is 
			characterized by two properties:
			\begin{itemize}
				\item[(i)]
				for every 
				$z \in [x_i + 1, x_i + \psi _{x _ i} ] \cap \N  $, 
				\begin{equation*}
					x_{i+1} + \psi _{x _{i+1}} \geq z + \psi _z,
				\end{equation*}
				\item[(ii)]
				and for every $z' \in [x_{i+1} + 1, x_i + \psi _{x _ i} ] \cap \N  $,
				\begin{equation*}
					x_{i+1} + \psi _{x _{i+1}} > z' + \psi _{z'}
				\end{equation*}
			\end{itemize}
			(here $[a,b] = \varnothing$ if $a > b$).
			By construction, $x _{i + 1} \leq x _ i + \psi _{x _ i}$, so 
			the left inequality in \eqref{slink} holds.
			A.s.\ on $\{ x \xrightarrow{ \Z _+ } \infty \}$, 
			$x_{i+1} + \psi _{x _{i+1}} >  x_i + \psi _{x _ i}$, because otherwise
			$x_i + \psi _{x _ i} + 1$ would not be wet.
			Hence a.s.\ on $\{ x \xrightarrow{ \Z _+ } \infty \}$ also $x_{i+2} + \psi _{x _{i+2}} >  x_{i+1} + \psi _{x _{i+1}} $. 
			Therefore the inequality  $    x _{i + 2} \leq x _ i + \psi _{x _ i} $
			is impossible a.s.\ on $\{ x \xrightarrow{ \Z _+ } \infty \}$ because it would contradict to (i) with $z = x _{i + 2}$.
		\end{proof}
	}
	
	\section{Notation, preliminaries, and further discussion}\label{sec proof outline}

	For each $x\in\Z$ and $j\in\N$, we denote by $\Viktor{ (S_t^{(x,j)})_{t \geq 0} }$ a simple
	symmetric continuous-time 
	random walk starting at $S_0^{(x,j)}=0$.
	We assume that the collection
	\begin{displaymath}
		\{S_t^{(x,j)},x\in\Z,j\in\N\}
	\end{displaymath}
	is i.i.d. For $m,n\in\N$, denote $\overline{m,n}=[m,n]\cap\Z$.
	For $t \geq 0$, $x \in \Z$, and $j \in \N$,
	the number
	$x + S_t^{(x,j)}$ is the position 
	of $j$-th particle started at location $x$,
	$t$ units of time after the sleeping particles
	at $x$ were activated.
	Let $(S_t, t \geq 0)$
	be a simple continuous-time random walk
	on $\Z$ and $\tau _k$
	be the $k$-th jump of $(S_t, t \geq 0)$,
	$\tau _0 = 0$.

	For two series $\sum\limits _n a_n $ and $\sum\limits _n b_n $
	with non-negative elements
	we write $\sum\limits _n a_n \simeq  \sum\limits _n b_n  $
	if they have the same convergence properties, that is,
	they either both converge or both diverge. 
	We write $\sum\limits _n a_n \precsim  \sum\limits _n b_n  $
	if  $\sum\limits _n b_n $ diverges, or if 
	both  $\sum\limits _n a_n $ and
	$\sum\limits _n b_n $ converge.
	This is true for example if $a_n \leq b_n$
	for large $n \in \N$ (but not necessarily for all $n \in \N$).
	We say that two events $A$ and $B$ are equal a.s., or coincide a.s., if  $\1 _A = \1 _B$ holds a.s.
	Multiplication takes precedence over
	taking maximum and minimum: 
	for $a,b,c \in \R$, $ab\vee c = (ab)\vee c$,
	$ab\wedge c = (ab)\wedge c$.
	
	As an auxiliary tool we consider the following construction of a TADIBP. Recall that $\{S_t^{(x,j)}\}$ are the random walks assigned to individual particles in the frog model with initial configuration \Peter{$(\eta(x))_{x\in\Z}$}, and let $A\colon \N \to (0, \infty)$ be a non-decreasing function. We define the random variables
	
	\begin{equation}\label{secrete}
		\ell ^{(A)} _x = \max \Big\{ k \in \Z _+:
		\exists t >0, j \in \overline{1, \eta (x)} \text{ such that } 
		t \leq \sum\limits^{x + S _{t} ^{(x ,j)}}_{z = x+1 } \frac{1}{A(z)} \text{ and } S _{t} ^{(x ,j)} \geq k  \Big\} \vee 0.
	\end{equation}
	(here as usual $\max \varnothing = -\infty $).

	We consider TADIBP with $\psi _x = \ell^{(A)}_x$.
	Heuristically, sites $x\in\Z$ which are wet in the TADIBP model are traversed by frogs at speed no less than $A(x)$. Therefore, if
	\eqref{A_series_converges} holds and (almost)
	all sites of the TADIBP are wet, it means that frogs traverse the space $\Z$ at high speed, leading to explosion of the system. Conversely,
	\eqref{A_series_diverges} and many
	dry sites imply that the frog model travels at low speed, leading to non-exploding expansion.

	Since $A$ is non-decreasing, we have
	\begin{displaymath}
		\fP\{ \ell ^{(A)}_x\geq k \}\geq\fP\{\ell ^{(A)}_{x+1}\geq k\},
		\ \ \ x \in \N, k \in \Z_+.
	\end{displaymath}

	\begin{rmk}
		The random variable
		$\ell ^{(A)}_x$ can be seen as the maximal distance travelled to the right
		by a particle 
		starting from $x$ 
		at a speed exceeding the given (varying) speed $A$.
	\end{rmk}

	The following elementary lemma is used throughout the paper.
	In particular, it can be applied to the Poisson distribution.
	\begin{lem} \label{reticent}
		Assume that for a sequence of positive numbers $\{\alpha _j\}_{j \in \N}$ there
		exist  $r \in (0,1)$ and $n \in \N$ such that 
		either for all $i \geq n$ 
		\begin{equation}\label{romp}
			\frac{\alpha _{i+1}}{\alpha_i} \leq r
		\end{equation}
		or for all $ i \in \N $
		\begin{equation}\label{romp2}
			\alpha _{n+i} \leq r ^i \alpha_n.
		\end{equation}
		Then there exists $C_{n, r}>1$ such that
		for $m \in \N$
		\begin{equation}\label{645572235432}
			\sum\limits _{i = m} ^\infty \alpha _i  \leq C_{n, r} \alpha _m.
		\end{equation}
		The constant $C_{n, r}$ can be chosen to depend only on $r$ and $n$.
	\end{lem}
	\begin{proof}
		Note that \eqref{romp} implies \eqref{romp2},
		and by \eqref{romp2} 
		\[
		\sum\limits _{i = n}^\infty \alpha _i \leq \sum\limits _{i = n}^\infty r^n\alpha _n
		=\frac{\alpha _n}{1-r}.
		\]
		Therefore $\sup\limits _{m \in \N} \frac{\sum\limits _{i = m} ^\infty \alpha _i}{\alpha _m} < \infty$, that is, \eqref{645572235432} holds
		for some $C>0$.
	\end{proof}
	
	In the article \cite{frogL} the authors described conditions 
	for the distinction between linear and superlinear spread. To this end, for a fixed $B>0$ they used the family $(\psi_x)_x$ given by
	\begin{displaymath}
		\psi_x=\max\set{y\geq x\colon\exists j,\exists t\colon\frac{S_t^{(x,j)}}{t}\geq B,S_t^{(x,j)}\geq y-x}.
	\end{displaymath}
	The expression 
	on the right hand side
	coincides 
	with our definition
	of 
	$\ell ^{(A)}_x$ in \eqref{secrete}
	with
	$A(x) \equiv B$.
	
	These random variables give 
	rise to totally asymmetric 
	discrete (homogeneous)
	Boolean percolation 
	as described in Section 
	\ref{sec tadibp}.
	The 
	proofs in \cite{frogL}
	rely on the following 
	statements.

	\begin{itemize}
		\item If percolation occurs for every constant $B>0$, then the spread is superlinear.
		\item If a positive fraction 
		of sites is dry for some $B>0$, then the spread is linear.
	\end{itemize}
	The conditions implying
	that percolation occurs,
	or that it does not occur,
	are then given in terms of 
	the distribution 
	of the initial number of 
	particles $\mu$.

	In this paper, the goal is to describe conditions separating the non-explosion and explosion as opposed to the linear and superlinear spread. The idea is to modify the family $(\psi_x)_x$
	so that 
	$\psi _x = \ell ^{(A)}_x$
	with an increasing function $A$
	as defined in \eqref{secrete}. Since $A$ is the ``speed'' at which the process propagates,
	the following statements should hold.
	
	\begin{itemize}
		\item If percolation 
		occurs for some $A$ with
		\begin{displaymath}
			\sum\frac{1}{A(x)}<\infty,
		\end{displaymath}
		then the process explodes.
		\item If percolation 
		does not
		occur for some $A$ with
		\begin{displaymath}
			\sum\frac{1}{A(x)}=\infty,
		\end{displaymath}
		then the process does not explode.
	\end{itemize}

	\begin{rmk}
		Note the similarity to ODE: Given
		\begin{displaymath}
			\dot{x}=f(x),  x(0) = 1,
		\end{displaymath}
		where $f$ is a non-negative continuous
		increasing
		function,
		we have explosion in finite time if
		\begin{displaymath}
			\int_1^\infty\frac{dy}{f(y)}<\infty.
		\end{displaymath}
	\end{rmk}
	
	The proof of explosion
	(Theorem \ref{main_theorem} (ii))
	closely follows the scheme we have just outlined.
	In Section \ref{sec explosion proof},
	we first show that percolation
	with $\psi _x = \ell ^{(A)}_x$
	implies explosion, and then proceed
	to establish that percolation 
	occurs a.s.\ under 
	assumptions in Theorem \ref{main_theorem} (ii).
	In contrast,
	when considering non-explosion 
	we do not directly
	rely on percolation not occurring,
	because a possible long range dependence
	makes it difficult to deduce non-explosion
	from non-percolation.
	Instead, we
	show that the
	assumptions in Theorem \ref{main_theorem} (i)
	imply
	${\inf\limits _{x \in \N}\P \{ x \text{ is dry} \} >0}$,
	and the latter is then shown to be incompatible with explosion.

	\section{Proof of non-explosion}\label{sec non expl}

	In this section, we prove the first part of Theorem \ref{main_theorem}.
	\Viktor{Here at the beginning of the section we lay out the roadmap of the proof.
		We first show that the explosions
		in two directions, $+\infty$ 
		and $-\infty$, are equivalent a.s.
		This is formulated precisely in \eqref{tuechtig = efficient}.}
	Because of that, it suffices to rule out the possibility of the explosion
	in direction $+\infty$ to prove non-explosion of the system. To this end, we show next that \Viktor{under conditions \eqref{A_series_diverges}
		and \eqref{skeevy = unpleasant, creepy}
	} the particles
	left to the origin
	cannot contribute to an explosion in this direction
	(this is formulated precisely in \eqref{profundity = being profound})
	and can thus be removed.
	\Viktor{After that totally asymmetric discrete inhomogeneous Boolean percolation  enters the picture. It is introduced in Section \ref{sec tadibp} with $\ell _x ^{(A)}$ defined 
		in \eqref{secrete}.
		In Proposition \ref{sleuth = detective}
		the probability that a site is dry is shown to be 
		separated from zero:
		$$
		\inf\limits _{m \in \N} \P\{ m \ \mathrm{is\ dry} \}  > 0.
		$$
		The final stretch of the proof of Theorem \ref{main_theorem}, (i),
		starts on Page~\pageref{proof non explosion}. There the dry sites are shown to be `slow' in a certain sense, and that the inequality $
		{\inf\limits _{m \in \N} \P\{ m \ \mathrm{is\ dry} \}  > 0}
		$ is incompatible with explosion.}
	
	
	Recall that
	$a_i = \frac{i!}{\left(\AA(i)\right) ^i}$,
	and that
	in this section
	we work under the following 
	assumption on $A$ and $\mu$.
	\begin{con}\label{con non-expl}
		It holds that $\sum\limits^{\infty }_{z = 1 }\frac{1}{A(z)} = \infty$
		and 
		\begin{equation}\label{throw the baby out with the bathwater: lose sth u didnt want to}
			\sum\limits _{i=0} ^{\infty} \mu\left( \left[
			a_i,
			\infty \right)  \right) < \infty.
		\end{equation}
	\end{con}

	\begin{lem}\label{curmudgeon}
		The series $\sum\limits _{i=1} ^\infty \frac{1}{A(i)} \wedge \frac 1i$
		is divergent. 
	\end{lem}
	\begin{proof}
		By the Cauchy condensation test
		\begin{equation*}
			\sum\limits _{i=1}^\infty \frac{1}{A(i)} \simeq 
			\sum\limits _{n=1}^\infty \frac{2^n}{A(2^n)}.
		\end{equation*}
		We have 
		\begin{equation*}
			\sum\limits _{i=1}^\infty \frac{1}{A(i)} \wedge \frac 1i
			= \sum\limits _{n = 0} ^\infty 
			\sum\limits _{i=2^n} ^{2^{n+1}-1} \frac{1}{A(i)} \wedge \frac 1i
			\geq 
			\sum\limits _{n = 0} ^\infty 
			2^n \Big[ \frac{1}{A(2^{n+1})} \wedge \frac {1}{2^{n+1}}\Big]
			= \frac 12 \sum\limits _{n = 1} ^\infty 
			2^n \Big[\frac{1}{A(2^{n})} \wedge \frac {1}{2^{n}}\Big]:=S.
		\end{equation*}
		If the set $\mathcal{Q}:= \{n \in \N: \frac{1}{A(2^{n})} \geq \frac {1}{2^{n}} \}$
		is infinite, 
		then $S \geq  \frac{1}{2}\sum\limits _{n \in \mathcal{Q}} 1 =\infty $. 
		If $\mathcal{Q}$ is finite, 
		then 
		\[
		S \simeq \sum\limits _{n = 1} ^\infty 
		2^n \frac{1}{A(2^{n})} \simeq \sum\limits _{i=1}^\infty \frac{1}{A(i)} 
		=\infty.
		\]
	\end{proof}

	Without loss of generality, 
	we can replace $A(i)$
	with $A(i) \vee i$: 
	indeed,  the series $\sum\limits _{i=1} ^\infty \frac{1}{A(i) \vee i}$
	is divergent by Lemma \ref{curmudgeon},
	and \eqref{skeevy = unpleasant, creepy}
	holds too since $\AA$ decreases if
	we make $A$ greater.
	Thus, we  assume henceforth that $A(i) \geq i$,
	$i \in \N$.

	Define \label{page: explosion occurs in both directions}
	$$\sigma ^r _\infty = \inf\{t \geq 0: \sup \A _t = \infty \} = 
	\inf\{t \geq 0: \text{no sleeping particles left on  } [0, \infty) \}$$
	and 
	$$\sigma ^l _\infty = \inf\{t \geq 0: \inf \A _t = -\infty \} = 
	\inf\{t \geq 0: \text{no sleeping particles left on  } (-\infty,0] \}.$$
	In this paper we do not investigate the question
	under which conditions a.s.\
	$ \sigma ^r _\infty = \sigma ^l _\infty $.
	However, we note here that 
	both events 
	$\{ \sigma ^r _\infty < \infty\}$
	and $\{ \sigma ^l _\infty < \infty\}$
	are tail events with respect to the \Viktor{sequence of} $\sigma$-algebras
	$\{\mathscr{F}_n \}_{n\in\N}$ \Viktor{defined in \eqref{sich in die Sonne liegen}.}
	Hence 
	$$\P \{ \sigma ^r _\infty < \infty\} \in \{0,1\} \ \ 
	\text{and} \ \ \P \{ \sigma ^l _\infty < \infty\} \in \{0,1\}.$$
	By symmetry it follows that 
	$\P \{ \sigma ^r _\infty < \infty\} = 
	\P \{ \sigma ^l _\infty < \infty\}$
	and hence the events 
	$\{ \sigma ^r _\infty < \infty\}$
	and $\{ \sigma ^l _\infty < \infty\}$ coincide a.s.,
	that is,
	the equality
	\begin{equation}\label{tuechtig = efficient 12}
		\1 _{\{ \sigma ^r _\infty < \infty\}} = \1 _{\{ \sigma ^l _\infty < \infty\}}  
	\end{equation}
	holds a.s.
	\Viktor{Therefore to prove non-explosion it is enough to show that $\P \{ \sigma ^r _\infty < \infty\} = 0$, and}
	in the rest of the section we
	concentrate only on  $\sigma ^r _\infty $.
	\Viktor{As an aside not needed in this proof we point out that \eqref{tuechtig = efficient 12}
		justifies the discussion in the introduction on Page~\pageref{die Huelle = shell} about an equivalent definition of explosion in dimension one.}
	
	Note that $ \sigma ^r  _\infty < \infty$
	if and only if (a.s.) there exists 
	a sequence of pairs $\{(x_n, t_n)\}_{n \in \Z_+}$,
	where $x _n \in \Z$, $x_0 = 0$, $0 = t_0 < t_1 < t_2 < ...$,
	satisfying 
	\begin{itemize}
		\item the sleeping particles at $x_n$
		are activated at $t_n$
		by an active particle started from $x_{n-1}$, $n \in \N$,
		and
		\item
		$ \lim\limits _{n \to \infty}  t_n := t_{\infty} < \infty $.
	\end{itemize}
	
	\emph{A priori}  it may be that 
	infinitely many elements of
	$\{x_n\}_{n \in \Z_+}$
	are negative. 
	\Viktor{Over the next few pages we}
	show that under Condition \ref{con non-expl}
	this is  impossible (as formulated in Corollary \ref{skeeve out = being disgusted} and \eqref{profundity = being profound}).
	Recall that
	$a_i = \frac{i!}{\left(\AA(i)\right) ^i}$,
	$i \in \N$ and $a_0 = 0$,
	and set 
	and  $b_i = \mu((a_{i-1}, a_i])$, $b_1 = \mu([0, a_1])$. 
	\Viktor{The next two lemmas have an auxiliary character and are used later to bound certain series.}
	\begin{lem}\label{lemma peter wants a label here}
		There exists 
		$C_a> 1$ such that
		\begin{equation}\label{peter wants a label here}
			\sum\limits _{i=j}^\infty \frac{1}{a_i} \leq \frac{C_a}{a_j}, 
			\ \ \ j \in \N.
		\end{equation}
	\end{lem}
	\begin{proof}
		Recall that we have assumed $A(i)\geq i$
		for $i \in \N$, which we can do due to Lemma~\ref{curmudgeon}. Let $\varepsilon \in (0, 0.1)$.
		For large $n \in \N$
		\begin{equation*}
			\frac{A^{-1}(n+1)}{\sum\limits _{j=1}^n A^{-1}(j)} \leq \frac{\varepsilon}{ n}
		\end{equation*}
		and hence 
		\begin{equation}\label{twine = thick threads}
			\left[\frac{\AA (n+1)  }{\AA (n)}\right] ^n
			= \left[1 + \frac{A^{-1}(n+1)}{\sum\limits _{j=1}^n A^{-1}(j)} \right] ^n \leq \left[1 + \frac{\varepsilon}{ n} \right] ^n
			\leq e^\varepsilon.
		\end{equation}
		Consequently for large $n \in \N$
		\begin{equation*}
			\frac{a_{n+1}}{a_n} = \frac{(n+1)!}{\big(\AA (n+1)\big)^{n+1}}:
			\frac{n!}{\big(\AA (n)\big)^{n}}
			= \frac{n+1}{\AA (n+1)} \left[\frac{\AA (n)}{\AA (n+1)  }\right] ^n
			\geq \frac{n+1}{e^\varepsilon \AA (n+1)}.
		\end{equation*}
		It remains to note that
		$\AA (n) \leq \sum\limits _{j = 1} ^n \frac 1j \leq 2+\ln n$, $n \in \N$,
		and hence 
		$\frac{n+1}{\AA (n+1)} \xrightarrow{n \to \infty} \infty$. 
	\end{proof}
	
	\begin{lem} \label{paisley}
		Let $\{ \alpha _i \} _{i \in \N}$ be an increasing sequence of 
		natural numbers satisfying for some $c_{\alpha}>0$
		\begin{equation*}
			\sum\limits _{i=j}^\infty \frac{1}{\alpha_i} \leq \frac{c_{\alpha}}{\alpha_j}, 
			\ \ \ j \in \N,
		\end{equation*}
		and let $\beta _i =\mu((\alpha_{i-1}, \alpha_i])$, 
		$\beta _1 =\mu([0, \alpha_1])$. 
		Then
		\begin{equation}
			\sum\limits _{i=1} ^{\infty} \frac{1}{\alpha_i} 
			\sum\limits _{k: k \geq 0, k \leq  \alpha_i}  \mu (k)
			k \leq  c_{\alpha},
		\end{equation}
		and 
		\begin{equation} \label{naschkatze sein = have a sweet tooth}
			\sum\limits _{i=2} ^{\infty} \frac{1}{\alpha_i} 
			\sum\limits _{k: k \geq 0, k \leq  \alpha_i} \mu (k)
			k \leq  
			c_{\alpha} \frac{\alpha _1}{\alpha _2}+
			c_{\alpha} (1 - \beta _1).
		\end{equation}
	\end{lem}

	\begin{proof}
		Set $\alpha_0=0$. We have 
		\begin{multline*}
			\sum\limits _{i=1} ^{\infty} \frac{1}{\alpha_i} 
			\sum\limits _{k: k \geq 0, k \leq  \alpha_i}  \mu (k)k
			=
			\sum\limits _{i=1} ^{\infty}
			\frac { 1}   {\alpha_i} 
			\sum\limits _{j = 1} ^i
			\sum\limits_{ k= \alpha_{j-1}+1  } ^{\alpha_j}
			\mu(k)  k
			\\
			\leq 
			\sum\limits _{i=1} ^{\infty}
			\frac { 1}   {\alpha_i} 
			\sum\limits _{j = 1} ^i
			\beta_j \alpha_j
			\leq  \sum\limits _{j = 1} ^{\infty} \beta_j \alpha_j \sum\limits _{i=j} ^{\infty}\frac { 1}   {\alpha_i}
			\leq  \sum\limits _{j = 1} ^{\infty} \beta_j \alpha_j  \frac{c_{\alpha}}{\alpha_j}
			= c_{\alpha}\sum\limits _{j = 1} ^{\infty} \beta_j
			\leq c_{\alpha}.
		\end{multline*}
		Similarly
		\begin{multline*}
			\sum\limits _{i=2} ^{\infty} \frac{1}{\alpha_i} 
			\sum\limits _{k: k \geq 0, k \leq  \alpha_i}  \mu (k)k
			=
			\sum\limits _{i=2} ^{\infty}
			\frac { 1}   {\alpha_i} 
			\sum\limits _{j = 1} ^i
			\sum\limits_{ k= \alpha_{j-1}+1  } ^{\alpha_j}
			\mu(k)  k
			\\
			\leq 
			\sum\limits _{i=2} ^{\infty}
			\frac { 1}   {\alpha_i} 
			\sum\limits _{j = 1} ^i
			\beta_j \alpha_j
			\leq  \sum\limits _{j = 1} ^{\infty} \beta_j \alpha_j \sum\limits _{i=j\vee2} ^{\infty}\frac { 1}   {\alpha_i}
			\leq  \sum\limits _{j = 1} ^{\infty} \beta_j \alpha_j  \frac{c_{\alpha}}{\alpha_{j \vee 2}}
			= c_{\alpha} \beta _1 \frac{\alpha _1}{\alpha _2} + c_{\alpha}(1 - \beta _1),
		\end{multline*}
		which gives \eqref{naschkatze sein = have a sweet tooth}.
	\end{proof}
	
	Let $( N_t^{(x,j)}, t \geq 0  )$ be a Poisson process
	obtained from $( S_t^{(x,j)}, t \geq 0  )$ by making all jumps be $+1$:
	for $q> 0$ the number
	$N_q^{(x,j)}$ can be seen as the number of jumps of $( S_t^{(x,j)}, t \geq 0  )$
	before the time $q$. Clearly, a.s.\ $S_t^{(x,j)} \leq N_t^{(x,j)} $
	for all $x \in \Z$, $j \in \N$, $t \geq 0$. 
	Also,
	let 	 $(N_t, t \geq 0)$
	be the Poisson process 
	with the same jumps as $(S_t, t \geq 0)$.
	\Viktor{The next lemma is key in establishing that under Condition \ref{con non-expl} particles left to the origin cannot contribute to explosion.}
	
	\begin{lem}
		There exists an increasing sequence $\{d_q\} _{q \in \N}$
		satisfying
		\begin{equation*}
			\sum\limits _{q \in \N}  \P \big\{ \max \{N_q^{(x,j)} + x:
			x < 0, 1\leq j\leq \eta (x)\} \geq d_q \big\}
			< \infty.
		\end{equation*}
		Note that  the time $q$ takes discrete values here.
	\end{lem}
	\begin{proof}
		By Lemma \ref{reticent}
		for $t \geq 0$
		there exists $C_t >0$
		such that for all
		$m, i \in \N$ 
		
		\begin{equation} \label{awol}
			\P \{N_t \geq m + i\} 
			\leq C_t e^{-t}\frac{t^{m + i}}{(m + i) !}
			< C _t e^{-t} \frac{t^{m + i} e ^{m + i}}
			{(m + i) ^{m + i}}
			= C_t e^{-t} \left( \frac{te}{m + i} \right)^{m + i}.
		\end{equation}
		By Condition \ref{con non-expl}
		for $n \in \N$
		there exists $ \kappa _n \in \N$  such that 
		\begin{equation}\label{raffle}
			\sum\limits _{i=1}  ^\infty \mu[ a_{i+\kappa _n} , \infty)
			=
			\sum\limits _{i=\kappa _n+1}  ^\infty \mu[ a_i , \infty) < \frac{1}{2^n}.
		\end{equation}
		For 
		$q,i \in \N$
		set $c_{q,i} : = C_q e^{-q} \left( \frac{qe}{d_q + i} \right)
		^{d_q + i}$, where $C_q$ is the constant in \eqref{awol}, and
		choose the sequence $d_1, d_2, ...$,
		$d_q \geq 2^q$,
		in such a way that $c_{q+1,i} < \frac {1}{2q} c_{q,i} $,
		$c_{1,i} \leq \frac {1}{a_i}$,
		\begin{equation}\label{the onus is on you}
			a_{i+\kappa _q} < c_{q,i}^ {-1}, 
			\ \ \ i, q \in \N,
		\end{equation}
		\begin{equation}\label{groove}
			c_{q,1}  \sum\limits _{k: k \geq0, k < c_{q,1}^{-1}} \mu (k)  k \leq \frac{1}{2^q}, \ \ \text{ and } \ \ 
			\mu([ c_{q,1}^{-1}, \infty)) \leq \frac{1}{2^q}.
		\end{equation}
		The sequence $d_1, d_2, ...$ can be constructed successively: given $d_1,...,d_n$, 
		$d_{n+1}$ can be chosen large enough 
		to satisfy all the conditions.
		It is important
		for \eqref{the onus is on you}
		that 
		by Condition \ref{con non-expl}
		the asymptotic
		growth rate of $j \mapsto a_j$ is actually lower
		than that of $j \mapsto \left(\frac{d + j}{c}\right)^{j}$
		for constants $c,d> 0$: that is,
		for every $c,d >0 $ for large $j$
		\[
		a_j < \frac{j ^j}{\left(\AA(j)\right) ^j} < \left(\frac{d + j}{c}\right)^{j}.
		\]
		For \eqref{groove} it is important that 
		\[
		\lim\limits _{Q \to \infty}
		\frac 1Q \sum\limits _{k: k \geq0, k < Q}\mu (k)  k 
		\leq 
		\lim\limits _{Q \to \infty}
		\sum\limits _{k: k \geq0}\mu (k) \left[\frac k Q 
		\wedge 1\right]
		= 0.
		\]
		We have for $q \in \N$
		by \eqref{awol}
		\begin{align}
			\P \big\{ \max \{N_q^{(i,j)} + i:
			i < 0, 1\leq j\leq \eta (i)\} \geq d_q \big\}
			\leq & \sum\limits _{i=1}  ^\infty
			\sum\limits _{k = 0} ^\infty \mu (k)
			\Big[k
			\P \{N_q \geq d_q + i\} \wedge 1 \Big] \notag
			\\
			\leq  &
			\sum\limits _{i=1}  ^\infty
			\sum\limits _{k = 0} ^\infty \mu (k)
			\Big[k
			c_{q,i} \wedge 1 \Big]  \notag
			\\
			= &
			\sum\limits _{i=1}  ^\infty \mu[ c_{q,i} ^{-1} , \infty)
			+ 
			\sum\limits _{i=1}  ^\infty c_{q,i} 
			\sum\limits _{k: k \geq0, k < c_{q,i}^{-1}}  \mu (k)k.
			\label{lampoon}
		\end{align}
		Taking the sum over $q$ in \eqref{lampoon} we get 
		\begin{multline} \label{a day late and a dollar short}
			\sum\limits _{q \in \N}
			\P \big\{ \max \{N_q^{(i,j)} + i:
			i < 0, 1\leq j\leq \eta (x)\} \geq d_q \big\}
			\\
			\leq 
			\sum\limits _{q =1} ^\infty
			\sum\limits _{i=1}  ^\infty \mu[ c_{q,i} ^{-1} , \infty)
			+
			\sum\limits  _{q =1} ^\infty
			\sum\limits _{i=1}  ^\infty
			c_{q,i} 
			\sum\limits _{k: k \geq0, k < c_{q,i} ^{-1}} \mu (k) k.
		\end{multline}
		
		Our conditions on $d_q$ and $c_{q,i}$ now imply
		that both sums on the right hand side of \eqref{a day late and a dollar short} are finite. 
		The first  is finite 
		since
		by \eqref{raffle} and \eqref{the onus is on you}
		$$\sum\limits _{q =1} ^\infty
		\sum\limits _{i=1}  ^\infty \mu[ c_{q,i} ^{-1} , \infty)
		\leq
		\sum\limits _{q =1} ^\infty
		\sum\limits _{i=1}  ^\infty \mu[ a_{\kappa _q + i} , \infty)
		\leq \sum\limits _{q =1} ^\infty \frac{1}{2^q}.
		$$
		
		To show that the second sum on the right hand side in \eqref{a day late and a dollar short}
		is finite, we split the sum into two and apply 
		\eqref{groove}
		and Lemma~\ref{paisley} to the internal sums
		with $\alpha_i = \lfloor c_{q,i} ^{-1} \rfloor$. 
		In notation of Lemma~\ref{paisley}
		we can take  
		$c_\alpha = 2$ for every $q \in \N$, and 
		use that $\frac{\alpha _1}{\alpha _2} \leq d_q^{-1}\leq 2^{-q} $ and $\beta _1 \geq 1 - 2^{-q}$ by the second inequality in 
		\eqref{groove}. We have
		
		\begin{multline*}
			\sum\limits  _{q =1} ^\infty
			\sum\limits _{i=1}  ^\infty
			c_{q,i} 
			\sum\limits _{k: k \geq0, k < c_{q,i} ^{-1}} \mu (k) k
			= \sum\limits  _{q =1} ^\infty  
			c_{q,1} 
			\sum\limits _{k: k \geq0, k < c_{q,1} ^{-1}} \mu (k) k
			+
			\sum\limits  _{q =1} ^\infty
			\sum\limits _{i=2}  ^\infty
			c_{q,i} 
			\sum\limits _{k: k \geq0, k < c_{q,i} ^{-1}} \mu (k) k
			\\
			\leq \sum\limits  _{q =1} ^\infty  
			\frac{1}{2^q}
			+
			\sum\limits  _{q =1} ^\infty
			\sum\limits _{i=2}  ^\infty
			\frac{1}{\lfloor c_{q,i} ^{-1} \rfloor } 
			\sum\limits _{k: k \geq0, k \leq \lfloor c_{q,i} ^{-1} \rfloor 
			} \mu (k) k 
			\leq \sum\limits  _{q =1} ^\infty  
			\frac{1}{2^q} +  \sum\limits  _{q =1} ^\infty
			\Big(2\frac{1}{2^q} + 2 \frac{1}{2^q} \Big) < \infty.
		\end{multline*}
	\end{proof}
	
	By the above lemma a.s.\ only finitely many events
	$\{  \max \{N_t^{(x,j)}:
	x < 0, 1\leq j\leq \eta (x)\} \geq d_t \}$, 
	$t \in \N$, occur. In particular, we have
	\begin{cor}\label{skeeve out = being disgusted}
		A.s.\ for all $t > 0$
		\begin{equation} \label{aufstreben}
			\sup\limits _{x < 0, 1 \leq j \leq \eta (x)} \big( S_t^{(x,j)}
			+ x \big)
			< \infty.
		\end{equation}
	\end{cor}
	This means that a.s.\
	the particles at $-1,-2,...$
	do not contribute to
	the explosion toward $+\infty$.
	More precisely, 
	let us modify our process by removing 
	all the sleeping particles left to the origin
	at the beginning and then proceeding as usual. 
	For this modified process let $\theta_n$, $n \in \N$,
	be the moment when \Viktor{site} $n$ is visited by an active particle for the first time, and let $\theta _\infty = \lim\limits_{n \to \infty} \theta _n$.
	Then clearly a.s.\ $\sigma ^r _\infty \leq \theta _\infty$, 
	however in view of Corollary \ref{skeeve out = being disgusted},
	a.s. 
	\begin{equation}\label{profundity = being profound}
		\1_{ \{\sigma ^r _\infty = \infty \}}
		= \1_{ \{\theta  _\infty = \infty \} }.
	\end{equation}
	
	\Viktor{The equality \eqref{profundity = being profound} represents a major stepping stone in the proof of Theorem~\ref{main_theorem}, (i). It allows us to remove at time $t = 0$ all sleeping particles left to the origin. From here on out we only consider the modified process with particles left to the origin removed; equivalently, we set $\eta (y) = 0$ for $y < 0$.}

	Recall that $\ell ^{(A)} _i$ was defined in \eqref{secrete}.
	\Viktor{In the remaining part of the proof of Theorem~\ref{main_theorem}, (i),  we consider totally asymmetric discrete inhomogeneous Boolean percolation from Section~\ref{sec tadibp} on $\Z _+$ with $\psi _i = \ell ^{(A)} _i$, $i \in \Z _+$.}
	The probability that $m \in \N$ is dry is given by
	\begin{equation}\label{scintillating}
		\P\{ m \text{ is dry} \} = \P\{ Y_m = 0 \}
		=\prod\limits _{i=0} ^{m-1} \P \{ \ell ^{(A)} _i \leq m-i \}
		= \prod\limits _{i=0} ^{m-1} \left( 1 -  \P \{ \ell ^{(A)} _i > m-i \}
		\right).
	\end{equation}
	\Viktor{In the next few lemmas we work towards establishing that 
		$
		\inf\limits _{m \in \N} \P\{ m \ \mathrm{is\ dry} \}  > 0
		$; this is achieved in Proposition \ref{sleuth = detective}.}
	\Viktor{The next two lemmas, Lemma \ref{waif = forsaken or orphaned child} and Lemma \ref{crack me up}, are auxiliary tools 
		in the proof of Proposition~\ref{sleuth = detective}.}
	\begin{lem}\label{waif = forsaken or orphaned child}
		There exists $C>0$ such that for $i \in \Z_+$, $j \in \N$
		\begin{equation*}
			\P\left\{ \exists t > 0:
			t \leq \sum\limits^{i + S _{t} }_{z = i+1 } \frac{1}{A(z)} 
			{\ \rm{ and } \ } S _{t}   \geq j 
			\right\} \leq 
			C
			\exp\left\{-\AA(i, i+j) \right\}
			\frac{ \left( \AA(i, i+j) \right) ^j}{j!}.
		\end{equation*}
	\end{lem}
	\begin{proof}
		Recall that $(S_t, t \geq 0)$
		is a simple continuous-time random walk
		on $\Z$,  $\tau _k$
		is the $k$-th jump of $(S_t, t \geq 0)$,
		$\tau _0 = 0$,
		and  $(N_t, t \geq 0)$
		is the Poisson process 
		with jumps at $\tau _1, \tau _2, ...$.
		Note that 
		\begin{multline*}
			\P\left\{ \exists t > 0:
			t \leq \sum\limits^{i + S _{t} }_{z = i+1 } \frac{1}{A(z)} 
			{\ \rm{ and } \ } S _{t}   \geq j 
			\right\} 
			=	\P\left\{ \exists t > 0:
			t \leq \AA ({i + S _{t} }) - \AA (i) 
			{\ \rm{ and } \ } S _{t}   \geq j 
			\right\} 
			\\*
			\leq
			\P\left\{ \exists t > 0:
			t \leq \AA ({i + N _{t} }) - \AA (i)
			{\ \rm{ and } \ } N _{t}   \geq j 
			\right\}.
		\end{multline*}
		Now
		\begin{align}
			\notag
			\P\{ \exists t > 0:
			t & \leq \AA ({i + N _{t} }) - \AA (i)
			{\ \rm{ and } \ } N _{t}   \geq j 
			\} 
			= 
			\P\left\{ \exists n \in \N, n\geq j:
			\tau_n \leq \AA (i, {i + n }) 
			\right\}
			\\ \notag
			= &
			\P\left\{ \exists n \in \N, n\geq j:
			N _{\AA (i, {i + n }) } \geq n 
			\right\} \leq
			\sum\limits _{n = j } ^\infty 
			\P\left\{ 
			N _{\AA (i, {i + n }) } \geq n 
			\right\}
			\\ \notag
			= & \sum\limits _{n = j } ^\infty 
			\sum\limits _{k = n} ^\infty 
			e^{- \AA (i, {i + n }) } 
			\frac{ \big(\AA  (i, {i + n})\big) ^k }{k!}
			= \sum\limits _{k = j } ^\infty 
			\frac{1}{k!}  \sum\limits _{n = j } ^k e^{- \AA (i, {i + n }) } 
			\big(\AA  (i, {i + n})\big) ^k
			\\ \notag
			\leq &
			e^{- \AA (i, {i + j }) } 
			\sum\limits _{k = j } ^\infty 
			\frac{1}{k!}
			\sum\limits _{n = j } ^k
			\big(\AA  (i, {i + k})\big) ^k
			\\ \label{warunki = conditions}
			= &  e^{- \AA (i, {i + j }) } 
			\sum\limits _{k = j } ^\infty 
			\underbrace{
				\frac{(k-j+1)\big(\AA  (i, {i + k})\big) ^{k} }{k!}}_{s_k}.
		\end{align}
		Note that $\AA  (i, {i + k}) \leq \AA  (k) \leq \ln k +2$
		and similarly to \eqref{twine = thick threads}
		\[
		\frac{\big(\AA  (i, {i + k+1})\big) ^{k+1 } }
		{\big(\AA  (i, {i + k})\big) ^{k }}
		\leq \left[ 1 + \frac 1k \right] ^k \leq e.
		\]
		Therefore the sequence $\{s_k \}_{k \in \N}$ 
		satisfies the conditions of 
		Lemma \ref{reticent} 
		uniformly in $i$ and $j$: 
		\[
		\frac{s_k}{s_{k+1}} = \frac{k+1}{\AA  (i, {i + k+1})}
		\frac{k-j+1}{k-j+2} \left[ \frac{\big(\AA  (i, {i + k+1})\big)  }
		{\big(\AA  (i, {i + k})\big) } \right]^k \geq
		\frac{k+1}{2e(\ln k +2)}.
		\]
		We see that the  the convergence
		$\frac{s_k}{s_{k+1}} \xrightarrow{k \to \infty} \infty$ takes place uniformly in $i$ and $j$.
		Consequently there exists $C>0$ such that 
		for all $i,j \in \N$ 
		\begin{equation*}
			\sum\limits _{k = j} ^ \infty s_k \leq C s_j.
		\end{equation*}
		Therefore by \eqref{warunki = conditions}
		\begin{multline*}
			\P\left\{ \exists t > 0:
			t \leq \AA ({i + N _{t} }) - \AA (i)
			{\ \rm{ and } \ } N _{t}   \geq j 
			\right\} \leq 
			C e^{- \AA (i, {i + j }) } 
			s_j
			=
			C e^{- \AA (i, {i + j }) } 
			\frac{\big(\AA  (i, {i + j})\big) ^{j} }{j!}.
		\end{multline*}
	\end{proof}
	
	\begin{lem}\label{crack me up}
		Let $0 < \alpha _1 \leq \alpha _2 \leq ... $
		be a sequence of positive numbers such that
		$\alpha _n \xrightarrow{n \to \infty} \infty$,
		$\limsup \limits _{n \to \infty} \frac{\alpha _{n+1}}{\alpha _n} > 1$, 
		and for some $C_\alpha > 1$
		\begin{equation*}
			\sum\limits _{i=m} ^{\infty}  \frac{1}{\alpha_i}
			\leq \frac{C_\alpha}{\alpha _m}.
		\end{equation*}
		Then for $C \geq 1$
		\begin{equation*}
			\sum\limits _{i=1} ^{\infty} \sum\limits_{k=0}^\infty \mu(k) \left(1 \wedge k C
			\alpha_i^{-1} \right)  
			\simeq 
			\sum\limits _{i=1} ^{\infty} \sum\limits_{k=0}^\infty \mu(k) \left(1 \wedge k
			\alpha_i^{-1} \right).
		\end{equation*}
	\end{lem}
	\begin{proof}
		We have 
		\begin{align}
			\sum\limits _{i=1} ^{\infty} \sum\limits_{k=0}^\infty \mu(k) \left(1 \wedge k C
			\alpha_i^{-1} \right) & \leq
			\sum\limits _{i=1} ^{\infty} \sum\limits_{k: k \geq
				C \alpha_i }  \mu(k) + 
			\sum\limits _{i=1} ^{\infty} \sum\limits_{k: k <
				C \alpha_i  } \mu(k)k C
			\alpha_i^{-1} \notag
			\\  &
			=
			\sum\limits _{i=1} ^{\infty}
			\mu \left( \left[C \alpha_i, \infty \right) \right)
			+ C 
			\sum\limits _{i=1} ^{\infty} \sum\limits_{k: k <
				C \alpha_i  } \mu(k)k \alpha_i^{-1}.
			\label{hergeben}
		\end{align}
		Since $C\geq 1$
		\begin{equation}\label{annehmen}
			\sum\limits _{i=1} ^{\infty}
			\mu \left( \left[C \alpha_i, \infty \right) \right)
			\precsim 
			\sum\limits _{i=1} ^{\infty}
			\mu \left( \left[ \alpha_i, \infty \right) \right).
		\end{equation}
		The second sum in \eqref{hergeben} is always finite.
		Indeed, set $\alpha_0=0$ and $ \beta _i = \mu ((C\alpha _{i-1}, C\alpha _i])$,
		then 
		\begin{multline*}
			\sum\limits _{i=1} ^{\infty} \sum\limits_{k: k <
				C \alpha_i  } \mu(k)k \alpha_i^{-1}
			=
			\sum\limits _{i=1} ^{\infty} 
			\sum\limits_{j = 1} ^{i}
			\sum\limits_{k: C \alpha _{j-1}  \leq k <
				C \alpha_{j}  } \mu(k)k \alpha_i^{-1}
			\leq \sum\limits _{i=1} ^{\infty} 
			\sum\limits_{j = 1} ^{i} \beta_j	C \alpha_{j}
			\alpha_i^{-1} 
			\\
			= C 
			\sum\limits _{j=1} ^{\infty} 
			\beta_j \alpha_{j}
			\sum\limits_{i = j} ^{\infty}
			\alpha_i^{-1} 
			\leq C C_\alpha \sum\limits _{j=1} ^{\infty} 
			\beta_j \alpha_{j} \alpha_j^{-1} \leq C C_\alpha.
		\end{multline*}
		Thus \eqref{annehmen} yields
		\begin{equation*}
			\sum\limits _{i=1} ^{\infty} \sum\limits_{k=0}^\infty \mu(k) \left(1 \wedge k C
			\alpha_i^{-1} \right)  
			\precsim
			\sum\limits _{i=1} ^{\infty} \sum\limits_{k=0}^\infty \mu(k) \left(1 \wedge k
			\alpha_i^{-1} \right).
		\end{equation*}
		Since 
		\begin{equation*}
			\sum\limits _{i=1} ^{\infty} \sum\limits_{k=0}^\infty \mu(k) \left(1 \wedge k C
			\alpha_i^{-1} \right)  
			\geq
			\sum\limits _{i=1} ^{\infty} \sum\limits_{k=0}^\infty \mu(k) \left(1 \wedge k
			\alpha_i^{-1} \right),
		\end{equation*}
		the statement of the lemma follows.
	\end{proof}
	
	\begin{prop} \label{sleuth = detective}
		We have
		\begin{equation}\label{cognate}
			\sup\limits _{m \in \N} 
			\sum\limits _{i=0} ^{m-1}  \P \{ \ell ^{(A)} _i > m-i  \}
			< \infty
		\end{equation}
		and 
		\begin{equation}\label{menagerie}
			\inf\limits _{m \in \N} \P\{ m \ \mathrm{is\ dry} \} 
			> 0.
		\end{equation}
	\end{prop}
	\begin{proof}
		The inequality $1-x \geq e^{-\frac{x}{1-x}}$, $x \in (0,1)$,
		implies
		\begin{equation}\label{assemble}
			\P\{ m \text{ is dry} \} 
			= \prod\limits _{i=0} ^{m-1} \left( 1 -  \P \{ \ell ^{(A)} _i > m-i \}
			\right)
			\geq \exp\left( - g_m \sum\limits _{i=0} ^{m-1}  \P \{ \ell ^{(A)} _i > m-i  \} \right),
		\end{equation}
		where 
		$$g_m =  \frac{1}{1 - \max\limits _{0 \leq i \leq m-1} \P \{ \ell ^{(A)} _i > m-i  \}}.$$
		Note that $g_m  \xrightarrow {m \to \infty} 1$
		since 
		$\max\limits _{0 \leq i \leq m-1} \P \{ \ell ^{(A)} _i > m-i  \}
		= \max\limits _{0 \leq i \leq m-1} \P \{ \ell ^{(A)} _{m-i} > i  \}
		\xrightarrow{m \to \infty} 0$
		which holds due to  $A(x)\xrightarrow{x \to \infty} \infty$.
		
		Let us find a bound for the sum in the exponent of \eqref{assemble}. 
		Conditioning on the number of particles on the site $i$ 
		we get 
		\begin{equation}\label{racy}
			\P \{ \ell ^{(A)} _i > m-i \} \leq 
			\sum\limits_{k=0}^\infty \mu(k) \left(1 \wedge k\P\left\{ \exists t > 0:
			t \leq \sum\limits^{i + S _{t} }_{z = i+1 } \frac{1}{A(z)} 
			\text{ and } S _{t}  \geq m-i 
			\right\}\right).
		\end{equation}	
		By \eqref{racy} and Lemma \ref{waif = forsaken or orphaned child}	for some $C\geq1$,	
		\begin{align}
			\notag
			\sum\limits _{i=0} ^{m-1}  \P \{ \ell ^{(A)} _i > m-i  \}
			& \leq  \sum\limits _{i=0} ^{m-1} \sum\limits_{k=0}^\infty \mu(k) \left(1 \wedge k\P\left\{ \exists t > 0:
			t \leq \sum\limits^{i + S _{t} }_{z = i+1 } \frac{1}{A(z)} 
			\text{ and } S _{t}  \geq m-i 
			\right\}\right)
			\\  & \notag
			\leq 
			\sum\limits _{i=0} ^{m-1} \sum\limits_{k=0}^\infty \mu(k) \left(1 \wedge k C
			\exp\left\{-\AA(i, m) \right\}
			\frac{ \left( \AA(i, m) \right) ^{m-i}}{(m-i)!}  \right)
			\\ & \notag
			\overset{i \to m -i}{=} 
			\sum\limits _{i=1} ^{m} \sum\limits_{k=0}^\infty \mu(k) \left(1 \wedge k C
			\exp\left\{-\AA(m-i, m) \right\}
			\frac{ \left( \AA(m-i, m) \right) ^{i}}{i!} \right)
			\\ &
			\leq
			\sum\limits _{i=1} ^{\infty} \sum\limits_{k=0}^\infty \mu(k) \left(1 \wedge k C
			\exp\left\{-\AA(m-i, m) \right\}
			\frac{ \left( \AA(m-i, m) \right) ^{i}}{i!} \right).
			\label{hand over fist}
		\end{align}
		Recall that 
		$a_i = 
		\frac{i!}{\left(\AA(i)\right) ^i}$.
		Since $A$ is non-decreasing
		for $i \in \Z_+ $ and $m \in \N$, $m > i$, we have
		\begin{equation}\label{thimble = cap for finger}
			a_i^{-1} = 
			\frac{ \left(\AA(i)\right) ^i}{i!}
			\geq 
			\exp\left\{ - \AA(m-i,m) \right\}
			\frac{ \left( \AA(m-i,m)  \right) ^{i}}{i!}.
		\end{equation}
		Hence by \eqref{hand over fist}
		and Lemma  \ref{crack me up} 
		
		\begin{align*}
			\sum\limits _{i=0} ^{m-1}  \P \{ \ell ^{(A)} _i > m-i  \}
			&	\leq \sum\limits _{i=1} ^{\infty} \sum\limits_{k=0}^\infty \mu(k) \left(1 \wedge k C a_i^{-1} \right)
			\\
			&
			\precsim
			\sum\limits _{i=1} ^{\infty} \sum\limits_{k=0} \mu(k) \left(1 \wedge k
			a_i^{-1} \right)    
			\\  &
			=
			\sum\limits _{i=1} ^{\infty} \sum\limits_{k: k \geq
				a_i } \mu(k) \left(1 \wedge k
			a_i^{-1} \right)     + 
			\sum\limits _{i=1} ^{\infty} \sum\limits_{k: k <
				a_i  } \mu(k) \left(1 \wedge k
			a_i^{-1} \right)   
			\\&
			= 
			\sum\limits _{i=1} ^{\infty} \mu\left( \left[
			a_i,
			\infty \right)  \right) + 
			\sum\limits _{i=1} ^{\infty} \sum\limits_{k: k <
				a_i   }
			\mu(k)  k
			a_i^{-1}
			\\  &
			= S_1 + S_2.
		\end{align*}
		
		By \eqref{throw the baby out with the bathwater: lose sth u didnt want to}
		\begin{equation}
			S_1 \leq  
			\sum\limits _{i=0} ^{\infty} \mu\left( \left[
			a_i,
			\infty \right)  \right) < \infty.
		\end{equation}
		To bound $S_2$ 
		recall that
		$a_0 = 0$
		and $b_i = \mu((a_{i-1}, a_i])$. 
		By Lemma \ref{lemma peter wants a label here}
		for some
		$C_a> 1$ we have
		\begin{equation*}
			\sum\limits _{i=j}^\infty \frac{1}{a_i} \leq \frac{C_a}{a_j}, 
			\ \ \ j \in \N,
		\end{equation*}
		and hence 
		\begin{multline*}
			S_2 = 	\sum\limits _{i=1} ^{\infty}
			\sum\limits _{j = 1} ^i
			\sum\limits_{k: a_{j-1} < k < a_j } 
			\mu(k)  k
			\frac { 1}   {a_i} 
			\\
			\leq 
			\sum\limits _{i=0} ^{\infty}
			\sum\limits _{j = 1} ^i
			b_j a_j
			\frac { 1}   {a_i}
			\leq  \sum\limits _{j = 1} ^{\infty} b_j a_j \sum\limits _{i=j} ^{\infty}\frac { 1}   {a_i}
			\leq  \sum\limits _{j = 1} ^{\infty} b_j a_j  \frac{C_a}{a_j}
			= C_a \sum\limits _{j = 1} ^{\infty} b_j \leq C_a.
		\end{multline*}
		Thus \eqref{cognate} is proven.
		Since $g_m \xrightarrow{m \to \infty} 1$,
		\eqref{menagerie} follows from
		\eqref{cognate} and 
		\eqref{assemble}.
	\end{proof}

	\begin{proof}[Proof of Theorem \ref{main_theorem}, (i).]
		\label{proof non explosion}
		Recall that 
		$\A_t$ is the set of sites visited by 
		\Viktor{active particles} by the time $t$,
		\begin{equation*}
			\theta _n  = \min\{t \geq 0: n \in \A _t  \},\quad n\in\N,
		\end{equation*}
		$\theta _\infty = \lim\limits _{n \to \infty} \theta  _n$,
		and all sleeping particles left to the origin are removed
		at the beginning.
		The event $\{ \theta _\infty < \infty\}$
		is a tail event with respect to  the $\sigma$-algebra
		$ \sigma \{S _{t} ^{(x ,j)}, t\geq 0, 0\leq x\leq n,
		1 \leq j \leq \eta (x) \}$.
		Hence 
		\begin{equation}\label{der Verdacht}
			\P\{ \theta _\infty < \infty\} \in \{0,1\}.
		\end{equation} 
		We have a.s. 
		\begin{equation*}
			\{ \text{explosion occurs} \} = \{ \theta _\infty < \infty\}.	 
		\end{equation*}
		\Viktor{ Let us now point out that  for every site $x \in \N$ there is a finite sequence $(y_j, k_j, s_j)_{j \in \{0,1,...,m\}}$ 
			such that $0 = y_0 < y_1 < ... <y_m = x$, and sleeping particles at $y_j$ are activated at time $s_j$ by the particle $(y_{j-1}, k_{j-1})$ started at $y_{j-1}$, $j = 1,...,m$. Furthermore, for every $x \in \N$ such a sequence is a.s.\ uniquely defined.}
		
		\Viktor{Consider a random sequence of particles  $(x_j, k_j)_{j \in { \mathcal{I}}}$
			such that 
			the site $x_{j+1} \in \Z_+$ 
			is   activated by
			the particle $(x_j, k_j)$, $1 \leq k_j \leq \eta (x_j)$, and  $x_0 = 0$.
			Denote also by $t_j$ the (random) time 
			when the site $x_j$ was activated.}
		The index set $\mathcal{I}$ is either $\Z_+$ 
		or $\{0,1,...,m\}$ for some $m \in \N$. 
		Note that we have
		$0= x_0 < x_1 < x_2 < ...$ 
		because
		all particles left of the origin are removed,
		and because we know exactly the order
		in which the sites of $\Z_+$
		are getting visited by active particles.
		We call the interval $(x_n, x_{n+1}]$
		fast if
		\begin{equation}\label{hackneyed = overused}
			t_{n+1} - t_n \leq  \sum\limits _{j= x_n +1} ^{x _{n+1}}\frac{1}{A(j)};
		\end{equation}
		otherwise we call the interval $(x_n, x_{n+1}]$
		slow. 
		\Viktor{Note that while it is not necessarily true 
			that every wet site belongs to a
			fast interval, it is true that every dry site belongs to a slow interval. 
			Indeed, take $y \in \N$. Consider the event $\{y \text{ is dry} \}$.
			We are going to show that a.s.\ on this event $y$ belongs to a slow interval
			$(x_n, x_{n+1}]$ for some $n \in \N$.
			By definition  a.s.\ on this event
			\begin{equation*}
				z + \ell ^{(A)} _z < y, \ \ \  \forall  z < y.
			\end{equation*}
			Therefore by definition of $\ell ^{(A)} _z$
			a.s.\ on $\{y \text{ is dry} \}$
			\begin{equation}\label{zaeh = tough, chewy}
				\forall z <y \, \forall k \geq y-z \, \forall t \geq 0 \,
				\forall j \in \overline{1, \eta (z)}:
				S _{t} ^{(z ,j)} \geq k \Rightarrow t > 
				\sum\limits^{z + S _{t} ^{(z ,j)}}_{m = z+1 } \frac{1}{A(m)}.
			\end{equation}
			Now for $n \in \Z _+$ consider the event $\{ y \in (x_n,x_{n+1}] \text{ and } 
			y \text{ is dry}\}$. Taking $z = x_n$
			and $k = x_{n+1}-x_n$
			in \eqref{zaeh = tough, chewy} 
			we find that a.s.\ on $\{ y \in (x_n,x_{n+1}] \text{ and } 
			y \text{ is dry}\}$
			\[
			\forall t \geq 0 \,
			\forall j \in \overline{1, \eta (x_n)}:
			S _{t} ^{(x_n ,j)} \geq x_{n+1}-x_n \Rightarrow t > 
			\sum\limits^{x_n + S _{t} ^{(x_n ,j)}}_{m = x_n+1 } \frac{1}{A(m)},
			\]
			hence   a.s.\ on $\{ y \in (x_n,x_{n+1}] \text{ and } 
			y \text{ is dry}\}$
			\begin{equation}\label{tuechtig = efficient}
				\forall t \geq 0 \,
				\forall j \in \overline{1, \eta (x_n)}:
				t  \leq 
				\sum\limits^{x_{n+1}}_{m = x_n+1 } \frac{1}{A(m)}
				\Rightarrow
				S _{t} ^{(x_n ,j)} < x_{n+1}-x_n.
			\end{equation}
			By construction the site $x_{n+1}$ 
			is activated at time $t_{n+1}$ by a particle started at  $x_n$,
			and the site $x_n$ was activated at time $t_n$. Therefore a.s.
			$$
			t_{n+1} - t_n = \inf\Big\{ t> 0 \big| \exists j \in \overline{1, \eta (x_n)}: 
			S_t ^{(x_n ,j)} = x_{n+1}-x_n.
			\Big\}
			$$
			From \eqref{tuechtig = efficient}  and the right-continuity of the random walk trajectories it follows that a.s.\ on $\{ y \in (x_n,x_{n+1}] \text{ and } 
			y \text{ is dry}\}$
			$$
			t_{n+1} - t_n  > \sum\limits^{x_{n+1}}_{m = x_n+1 } \frac{1}{A(z)},
			$$
			that is, that the $(x_n, x_{n+1}]$ is slow.
			Taking a union over $n$ we see that indeed a.s.\ every dry sight belongs to a slow interval.
		}
		
		\Viktor{Since particles at the site $x_{n+1}$ are activated
			by a particle started at $x_n$ we have 
			$    \theta _{x_{n }} = t _{n} 
			$.}
		\Viktor{Next we} bound $ \theta _{x_n}$
		from below
		by imagining that fast intervals
		are traveled over instantaneously,
		whereas slow intervals take 
		the time equal to 
		the expression on the right hand side of 
		\eqref{hackneyed = overused} to traverse.
		\Viktor{By definition of a slow interval we have for $m \in \N $ a.s.
			\begin{multline*}
				\theta _{x_m} = t _m  = (t _m - t _{m-1}) + 
				(t _{m-1} - t _{m-2}) + ... + (t _1 - t _0) 
				\\
				\geq \sum\limits _{j= 1} ^{x _{m}}\frac{1}{A(j)} \1\{j \text{ belongs to a slow interval} \}.
			\end{multline*}
		}
		\Viktor{Since a.s. every dry sight belongs to a slow interval we also have
			a.s.
			\begin{equation*}
				\theta _{x_m} 
				\geq \sum\limits _{j= 1} ^{x _{m}}\frac{1}{A(j)} \1\{j \text{ is dry} \}.
			\end{equation*}
			and for $n< m$
			\begin{equation}\label{putative = considered, reputed to be}
				\theta _{x_m} - \theta _{x_n} 
				\geq \sum\limits _{j= x_n+1} ^{x _{m}}\frac{1}{A(j)} \1\{j \text{ is dry} \}.
			\end{equation}
		}
		By  Proposition \ref{sleuth = detective}
		for some $c \in (0,1]$
		\begin{equation}\label{gimpy = limping, lame}
			\P(D_n) \geq c, 
			\ \ \ n \in \N,
		\end{equation}
		where $D_n = \{ n \text{ is dry}\}$.
		For $n \in \N$ let $r_n$ be the minimal element 
		of
		$\{x_j\}_{j \in \Z_+}$ to the right of $n$:
		\begin{equation*} 
			r_n = \min\{x: x \in \{x_j\}_{j \in \Z_+}, x \geq n \}.  
		\end{equation*} 
		Assume that 
		\begin{equation}\label{munchkin = gnom}
			\P\{ \theta _\infty < \infty\} = 1.
		\end{equation}
		\Viktor{Since a.s.\ $\theta_m \to \theta_\infty$, $m \to \infty$, 
			there exists $N \in \N$ such that the event $B := \{\theta _\infty \leq \theta _{_N} +1 \} $ satisfies $P(B)\geq 1 - \frac c3$.	
			Note that a.s.\ on $B$}
		\begin{equation}\label{salacious}
			\theta_\infty \leq \theta_n + 1 \leq \theta_{r_n} + 1, \  \ \ n \geq N. 
		\end{equation}
		We have
		\[
		\P(D_i \cap B) \geq \P(D_i) + \P(B) - 1 \geq \frac 23 c.
		\]
		There exists $N' \in \N$, $N' \geq N$,
		such that 
		\[
		\P\{ r _{_N} \leq N' -1 \} \geq 1 - \frac c3.
		\]
		We have \Viktor{then} $\P(D_i \cap B \cap  \{ r _{_N} \leq N' \} )
		\geq \frac c3$, and 
		hence \Viktor{by \eqref{putative = considered, reputed to be}}
		\begin{align*}
			\E \big[ (\theta_\infty - \theta _{_N}) \1_B \big] & \geq 
			\E\big[ (\theta_\infty - \theta _{_{N'}}) \1_B 
			\1 \{ r _{_N} \leq N' \} \big]
			\\
			& \geq 
			\E \Big[ \1_B 
			\1 \{ r _{_N} \leq N' \} 
			\sum\limits _{i = N'}^{\infty}
			\1 \{j \text{ is dry} \} 
			\frac{1}{A(j)} \Big]
			\\
			& = 
			\sum\limits _{i = N'}^{\infty}
			\frac{1}{A(j)} \E \big[\1_B 
			\1 \{ r _{_N} \leq N' \} 
			\1 \{j \text{ is dry} \} \big]
			\\
			&= 
			\sum\limits _{i = N'}^{\infty}
			\frac{1}{A(j)} \P( B
			\cap 
			\{ r _N \leq N' \} 
			\cap D_j 
			)
			\geq
			\sum\limits _{i = N'}^{\infty} 
			\frac{c}{3 A(i)} =  
			\infty,
		\end{align*}
		but this contradicts \eqref{salacious}
		since by \eqref{salacious}
		it should hold that $\E (\theta_\infty - \theta _N)  \1_B \leq 1$. 
		Thus \eqref{munchkin = gnom} cannot hold,
		and we have \Viktor{by \eqref{der Verdacht}}
		\begin{equation}
			\P\{ \theta _\infty < \infty\} = 0,
		\end{equation}
		that is, the probability of explosion is zero.
	\end{proof}

	\section{Proof of explosion}
	\label{sec explosion proof}

	This section is devoted to the proof of explosion of the frog model. First we relate the associated TADIBP to the explosion of the frog model. Next we state conditions for percolation, and this gives the desired result.

	\subsection{Connecting TADIBP and explosion}

	As stated above, our first step is to relate percolation of the TADIBP to explosion of the frog model.
	This approach uses the activation times of certain sites related to the TADIBP process, similar to Proposition 4.2 in \cite{frogL}.

	\begin{prop}\label{perc_expl}
		Assume that the TADIBP for $\{\ell_x^{(A)}\}_{x\in\LZ_+}$ percolates, where $A$ satisfies
		\begin{displaymath}
			\sum_{z=1}^\infty\frac{1}{A(z)}<\infty.
		\end{displaymath}
		\Peter{Then for any $x_0\in\N$, explosion occurs almost surely on $\{x_0\xrightarrow{\LZ_+}\infty\}$.}
	\end{prop}
	\begin{proof}
		Consider TADIBP with $\psi _x = \ell^{(A)}_x$.
		Take a percolation sequence $\{x_n\}_{n\in\LN}$ given by Lemma~\ref{perc_sequence} \Peter{corresponding to the set $\{x_0\xrightarrow{\LZ_+}\infty\}$}. Set $y_n=x_n+\ell_{x_n}^{(A)}$ and denote by $\sigma_x$ the activation time of location $x$, i.e.\ the first time an active frog visits site $x$.
		Note that a.s.\ $\sigma_x < \infty$ for every $x \in \Z$
		since at $t = 0$ there is at least one active particle at the origin.
		\Viktor{By definition of $\ell_{x_n}^{(A)}$
			in \eqref{secrete} a.s.\ on
			$\{x_0\xrightarrow{\LZ_+}\infty\}$
			there exists $j_n\in\overline{1,\eta(x_n)}$
			such that}
		\begin{displaymath}
			\sigma_{y_n}-\sigma_{y_{n-1}}\leq\sigma_{y_n}-\sigma_{x_n}\leq\sum_{z=x_n+1}^{x_n+S_{\sigma_{y_n}}^{(x_n,j_n)}}\frac{1}{A(z)}.
		\end{displaymath}
		Furthermore, since each point $z$ is in at most two of the intervals $[x_i,x_i+\ell_{x_i}^{(A)}]$, we have that 
		\begin{displaymath}
			\sum_{n=1}^\infty(\sigma_{y_n}-\sigma_{y_{n-1}})
			\precsim  
			2\sum_{z=1}^\infty\frac{1}{A(z)}<\infty.
		\end{displaymath}
		Therefore
		$\sigma _\infty = \lim\limits _{n \to \infty} \sigma _n < \infty.$, i.e.\ the total activation time ``up to infinity'' is finite.
	\end{proof}

	\subsection{Conditions for percolation}
	
	The next step is to find conditions on the tail distribution of TADIBP
	with $\psi _x = \ell^{(A)}_x$ such that the system  percolates.
	The random variables 
	$\psi _x = \ell^{(A)}_x$
	are independent but not identically distributed. 
	Recall that the Markov chain $\{Y_m\}_{m \in \Z_+}$
	was defined in Definition \ref{Aufwand = Effort}.
	Note that $Y_m>0$ for all but finitely many $m\in\Z_+$ is equivalent to percolation of the system $\{\psi_x\}_{x\in\Z_+}$.
	
	\begin{lem}\label{conv_perc}
		Assume that 
		\begin{displaymath}
			\sum_{m=1}^\infty\prod_{i=0}^m(1-\fP\{\psi_{m-i}>i \})<\infty.
		\end{displaymath}
		Then a.s.\ there exists $x_0\in\LZ_+$ connected to $\infty$:
		\begin{displaymath}
			\fP\{x_0\xrightarrow{\LZ_+}\infty\text{ for some }x_0\in\Z_+\}=1.
		\end{displaymath}
	\end{lem}
	\begin{proof}
		Since all $\psi_x$ are independent, the following identity holds by Definition \ref{Aufwand = Effort}:
		\begin{displaymath}
			\fP\{ Y_m=0 \} =\prod_{i=0}^m(1-\fP\{ \psi_{m-i}>i\}).
		\end{displaymath}
		By Borel-Cantelli, the assumption implies that
		\begin{displaymath}
			\fP\{ Y_m=0\text{ infinitely often}\}=0,
		\end{displaymath}
		and the system percolates.
	\end{proof}

	Next, we need to find out which conditions on the initial distribution imply that
	\begin{equation}\label{tail_series}
		\sum_{m=1}^\infty\prod_{i=0}^m(1-r_{m-i}^i)<\infty,
	\end{equation}
	where we set $r_{m-i}^i=\fP \{ \psi_{m-i}>i \}$. To this end, we establish inhomogeneous analogues of the lemmas from \cite{frogL}. We write $r_{m-i}^i(A)$ if the coefficient corresponds to $\psi_{m-i}=\ell_{m-i}^{(A)}$.  Note that $i$ and $m-i$ are interchangeable in \eqref{tail_series}.

	The following lemma shows that we may assume that $A(m)>1$ for all $m\in\N$.
	\begin{lem}\label{A_gtr_1}
		Set
		\begin{displaymath}
			z_0:=\min\{z\in\N\colon A(z)>1\text{ and }\mu([0,A(z)])>0\}.
		\end{displaymath}
		Furthermore, define $\tilde{A}(m):=A(m+z_0-1)$. Then for any $\rho>1$, 
		\begin{displaymath}
			\sum_{m=1}^\infty\prod_{i=1}^m\mu([0,A(m)^{\rho i}])\simeq\sum_{m=1}^\infty\prod_{i=1}^m\mu([0,\tilde{A}(m)^{\rho i}]).
		\end{displaymath}
		Note that $\tilde{A}(m)>1$ for all $m\in\N$.
	\end{lem}
	\begin{proof}
		We show ``$\succsim$'' first, i.e., assume that
		\begin{displaymath}
			\sum_{m=1}^\infty\prod_{i=1}^m\mu([0,A(m)^{\rho i}])<\infty.
		\end{displaymath}
		We have
		\begin{align*}
			\sum_{m=1}^\infty\prod_{i=1}^m\mu([0,A(m)^{\rho i}])& \geq\sum_{m=z_0}^\infty\prod_{i=1}^m\mu([0,A(m)^{\rho i}])
			\\
			&=\sum_{m=1}^\infty\prod_{i=1}^{m+z_0-1}\mu([0,A(m+z_0-1)^{\rho i}])
			\\
			&=\sum_{m=1}^\infty\prod_{i=1}^{m+z_0-1}\mu([0,\tilde{A}(m)^{\rho i}])
			\\
			&=\sum_{m=1}^\infty\Big[\prod_{i=1}^{z_0-1}\mu([0,\tilde{A}(m)^{\rho i}])\Big]\Big[\prod_{i=z_0}^{m+z_0-1}\mu([0,\tilde{A}(m)^{\rho i}])\Big]
			\\
			&\geq\prod_{i=1}^{z_0-1}\mu([0,\tilde{A}(1)^{\rho i}])\sum_{m=1}^\infty\prod_{i=1}^m\mu([0,\tilde{A}(m)^{\rho i+\rho z_0-\rho}])
			\\
			&\geq\prod_{i=1}^{z_0-1}\mu([0,\tilde{A}(1)^{\rho i}])\sum_{m=1}^\infty\prod_{i=1}^m\mu([0,\tilde{A}(m)^{\rho i}]).
		\end{align*}
		For the direction ``$\precsim$'', assume that
		\begin{displaymath}
			\sum_{m=1}^\infty\prod_{i=1}^m\mu([0,\tilde{A}(m)^{\rho i}])<\infty.
		\end{displaymath}
		Then
		\begin{align*}
			\sum_{m=1}^\infty\prod_{i=1}^m\mu([0,A(m)^{\rho i}])&\leq z_0-1+\sum_{m=z_0}^\infty\prod_{i=1}^m\mu([0,A(m)^{\rho i}])
			\\
			&=z_0-1+\sum_{m=1}^\infty\prod_{i=1}^{m+z_0-1}\mu([0,A(m+z_0-1)^{\rho i}])
			\\
			&=z_0-1+\sum_{m=1}^\infty\prod_{i=1}^{m+z_0-1}\mu([0,\tilde{A}(m)^{\rho i}])
			\\
			&\leq z_0-1+\sum_{m=1}^\infty\prod_{i=1}^{m}\mu([0,\tilde{A}(m)^{\rho i}]),
		\end{align*}
		which proves the second direction.
	\end{proof}
	
	In view of Lemma \ref{A_gtr_1}, we may assume from now on that $A(m)>1$ for all $m\in\N$.
	To estimate $r_{m-i}^i$, we need the following lemma, which is the inhomogeneous analogue to Lemma~4.6 of \cite{frogL}:
	Recall that $(S_t)_{t\geq 0}$ is a continuous-time simple random walk on $\Z$.
	\begin{lem}\label{eps_estimate}
		Assume that $A\colon\N\to(1,\infty)$ is non-decreasing and
		\begin{equation}\label{A_diverges}
			\lim_{z\to\infty}A(z)=\infty.
		\end{equation}
		
		Then for any $n,x\in\Z_+$,
		\begin{displaymath}
			\fP\Big\{\exists t\geq 0\colon t\leq\sum_{z=x+1}^{x+S_t}\frac{1}{A(z)}, S_t>n\Big\}\geq\frac{e^{(1-\frac{1}{A(x+n+1)})(n+1)}}{A(x+n+1)^{n+1}e2^{n+1}\sqrt{n+1}}.
		\end{displaymath}
		
	\end{lem}
	\begin{proof}
		Recall that $\tau _n$ 
		is the $n$-th jump of $(S_t, t\geq 0)$ and thus
		has the Erlang distribution as the sum of $n$ independent unit exponentials. In particular,
		\[
		\P \{ \tau _n \leq b\} \geq \frac{e^{-b}b^n}{n!}.
		\]
		{
			Since $\frac{1}{A(z)}<1$ for all $z\in\N$,
			\begin{align*}
				\fP\Big\{\exists t\geq 0\colon t\leq\sum_{z=x+1}^{x+S_t}\frac{1}{A(z)}, S_t>n\Big\}&\geq\fP\Big\{\exists t\geq 0\colon t\leq\frac{n+1}{A(x+n+1)}, S_t=n+1\Big\}
				\\
				&\geq\fP\Big\{\frac{\tau_{n+1}}{n+1}\leq A(x+n+1)^{-1}\Big\}\fP(S_{\tau_j}-S_{\tau_j-}=1, j\in\overline{1,n+1})
				\\
				&\geq e^{-\frac{n+1}{A(x+n+1)}}\frac{(n+1)^{n+1}}{A(x+n+1)^{n+1}(n+1)!}2^{-(n+1)}
				\\
				&\geq\frac{e^{(1-\frac{1}{A(x+n+1)})(n+1)}}{A(x+n+1)^{n+1}e2^{n+1}\sqrt{n+1}}.
			\end{align*}
		}
	\end{proof}
	
	For the convergence analysis below we set for $m\in\N$ and $i\in\{0,1,\dotsc,m\}$
	\begin{displaymath}
		E_2(i,m)=\frac{1}{e\sqrt{i+1}}\Big(\frac{e^{1-\frac{1}{A(m+1)}}}{2A(m+1)}\Big)^{i+1}.
	\end{displaymath}
	
	Recall that $r_{m-i}^i$ was defined right after \eqref{tail_series}.
	\begin{lem}[cf. Lemma 5.4, \cite{frogL}]\label{L5.4}
		Assume that $A(m)>1$ for all $m\in\N$ and
		\begin{displaymath}
			\lim_{z\to\infty}A(z)=\infty.
		\end{displaymath}
		Then for the function {$E_2(i,m)$} defined above
		we have
		\begin{displaymath}
			r_{m-i}^i(A) 
			\geq 1-\sum_{k=0}^\infty\mu(k)\left[1-E_2(i,m)\right]^k
		\end{displaymath}
		for all $m\in\N$ and $i\in\{0,1,\dotsc,m\}$.
	\end{lem}
	
	\begin{proof}
		We want to estimate $r_{m-i}^i(A)=\fP\{\ell_{m-i}^{(A)}>i\}$. To this end, note that 
		\begin{align*}
			&\{\ell_{m-i}^{(A)}>i\}^c=
			\\
			&=\bigcup_{k=0}^\infty\Bigg(\set{\eta(m-i)=k}\cap \bigg\{\forall t\geq 0\ \forall j\in\set{1,\dotsc,k}\colon t>\sum_{z=m-i+1}^{m-i+S_t^{(m-i,j)}}\frac{1}{A(z)}\text{ or }S_t^{(m-i,j)}\leq i\bigg\}\Bigg)
		\end{align*}
		where the union is disjoint. Since all random walks are independent, this in turn implies that
		\begin{displaymath}
			r_{m-i}^i(A)=1-\fP\{\ell_{m-i}^{(A)}\leq i\}=1-\sum_{k=0}^\infty\mu(k)\left(\fP\Big\{\forall t\geq\tau_i\colon t>\sum_{z=m-i+1}^{m-i+S_t}\frac{1}{A(z)}\text{ or }S_t\leq i\Big\}\right)^k.
		\end{displaymath}
		We may replace the condition $t>0$ by $t\geq\tau_i$ above, since it is impossible for the process $(S_t)_t$ to be larger than $i$ before the $i$-th jump (even before the $(i+1)$-st jump).
		By Lemma~\ref{eps_estimate}, we have for every $m$
		and $i \in \{0,1,...,m\}$ 
		\begin{align*}
			\fP& \Big \{ \forall t\geq\tau_i\colon t>\sum_{z=m-i+1}^{m-i+S_t}\frac{1}{A(z)}\text{ or }S_t\leq i\Big\}
			\\
			&=1-\fP\Big \{ \exists t\geq\tau_i\colon t\leq\sum_{z=m-i+1}^{m-i+S_t}\frac{1}{A(z)},S_t>i\Big \}
			\\
			&{\leq 1-\frac{e^{(1-\frac{1}{A(m+1)})(i+1)}}{A(m+1)^{i+1}e2^{i+1}\sqrt{i+1}}}
			{= 1 - E_2(i,m).}
		\end{align*}
	\end{proof}

		%

	\begin{rmk}\label{index_estimation}
		Fix $m_0,i_0\in\N$. We have
		\begin{align*}
			\sum_{m=1}^\infty\prod_{i=0}^m(1-r_{m-i}^i)&=\sum_{m=1}^{m_0-1}\prod_{i=0}^m\underbrace{(1-r_{m-i}^i)}_{\leq 1}+\sum_{m=m_0}^{\infty}\prod_{i=0}^m(1-r_{m-i}^i)
			\\
			&\leq m_0-1+\sum_{m=m_0}^\infty\prod_{i=i_0}^m(1-r_{m-i}^i).
		\end{align*}
		Therefore, for the convergence of \eqref{tail_series}, it suffices to consider
		\begin{displaymath}
			\sum_{m=m_0}^\infty\prod_{i=i_0}^m(1-r_{m-i}^i)
		\end{displaymath}
		for some $m_0,i_0\in\N$, or
		\begin{displaymath}
			\sum_{m=m_0}^\infty\prod_{i=i_0}^m(1-r_i^{m-i}),
		\end{displaymath}
		\Peter{because $i$ and $m-i$ are interchangeable and hence, the two expressions above are equal.}
	\end{rmk}
	
	Since we want to find conditions on the convergence of \eqref{tail_series}, by Lemma \ref{L5.4}, we may use the inequality
	\begin{displaymath}
		\sum_{m=m_0}^\infty\prod_{i=i_0}^m(1-r_{m-i}^i(A))\leq\sum_{m=m_0}^\infty\prod_{i=i_0}^m\sum_{k=0}^\infty\mu(k)(1-{E_2(i,m)})^k.
	\end{displaymath}
	Let us now bound {$E_2(i,m)$} from below
	by a quantity involving $A$. This bound will be used in the final part of the proof of explosion.
	

\begin{lem}\label{E1_est}
	Assume that the non-decreasing function $A\colon\N\to(1,\infty)$ fulfills
	\begin{displaymath}
		\sum_{z=1}^\infty\frac{1}{A(z)}<\infty.
	\end{displaymath}
	Then there exists $m_0\in\N$ such that for all $m\geq m_0$ and all $0\leq i\leq m$, 
	\begin{displaymath}
		E_2(i,m)\geq\left(\frac{1}{2A(m+1)}\right)^{i+2}.
	\end{displaymath}
\end{lem}

\begin{proof}
	First of all, note the following properties of $A$:
	\begin{itemize}
		\item $A$ diverges, i.e.\ \eqref{A_diverges} is fulfilled.
		\item There exists $m_0$ such that
		\begin{equation}\label{sqrt_estimate_A}
			\frac{1}{A(m+1)}\leq\frac{2}{e\sqrt{m+1}}
			\ \ \ \text{for all } m \geq m_0.
		\end{equation}
	\end{itemize}
	
	This can be seen as follows: assume that this is not the case. Then for each $n_0$, there exists $m\geq n_0$ such that
	\begin{equation}\label{contradiction_A}
		\frac{1}{A(m+1)}>\frac{2}{e\sqrt{m+1}}.
	\end{equation}
	Since $A$ is non-decreasing, we have for all $n\leq m$,
	\begin{displaymath}
		\sum_{j=1}^{m+1}\frac{1}{A(j)}\geq\sum_{j=1}^{m+1}\frac{1}{A(m+1)}>\frac{2(m+1)}{e\sqrt{m+1}}=\frac{2\sqrt{m+1}}{e}.
	\end{displaymath}
	Choosing a sequence $\{m_k\}_{k\in\N}$ such that $m_k\to\infty$ and \eqref{contradiction_A} holds for all $k\in\N$, we have
	\begin{displaymath}
		\lim_{k\to\infty}\sum_{j=1}^{m_k+1}\frac{1}{A(j)}=+\infty
	\end{displaymath}
	which contradicts the assumption of the lemma. Hence \eqref{sqrt_estimate_A} is true.

	Let $m\geq m_0$, where $m_0\in\N$ is chosen such that \eqref{sqrt_estimate_A} holds. Using that $e^{1-\varepsilon}\geq 1$ and  taking into account that $i\leq m$, we have
	
	\begin{align*}
		E_2(i,m)=\underbrace{\frac{1}{e\sqrt{i+1}}}_{\geq\frac{1}{e\sqrt{m+1}}}\underbrace{\Big(\frac{e^{1-\frac{1}{A(m+1)}}}{2A(m+1)}\Big)^{i+1}}_{\geq\Big(\frac{1}{2A(m+1)}\Big)^{i+1}}\geq\Big(\frac{1}{2A(m+1)}\Big)^{i+2}.
	\end{align*}
\end{proof}
Next, we want to translate the condition above into something more useful, i.e.\ dependent on the initial distribution of the frog model.

\begin{lem}\label{shrew is a small animal}
	We have
	\begin{equation}
		\sum_{m=m_0}^\infty\prod_{i=i_0}^m\sum_{k=0}^\infty\mu(k)(1-E_2(i,m))^k\precsim\sum_{m=m_0}^\infty\prod_{i=i_0}^m\mu([0,2A(m)^{{ {\rho}}i}]).
	\end{equation}
\end{lem}

\begin{proof}
	By Lemma \ref{E1_est}, we have
	\begin{align}
		\sum_{k=0}^\infty\mu(k)(1-{E_2(i,m)})^k&\leq\sum_{k=0}^\infty\mu(k)\left(1-\left(\frac{1}{2A(m+1)}\right)^{i+2}\right)^k \notag
		\\
		&\leq\sum_{k=0}^\infty\mu(k)\left(1-\left(\frac{1}{2A(m+2)}\right)^{i+2}\right)^k. \label{froth}
	\end{align}
	Since we may start the convergence analysis at arbitrary $m_0$ and $i_0$, we may shift the index $m+2\mapsto m$ and $i+ 2\mapsto i$ and consider
	\begin{displaymath}
		\sum_{k=0}^\infty\mu(k)\left(1-\left(\frac{1}{2A(m)}\right)^{i}\right)^k.
	\end{displaymath}
	\label{page: reduction of exponent}
	Set $\varkappa=\rho-1>0$.
	We have
	\begin{align}
		\sum_{k=0}^\infty\mu(k)\left(1-2A(m)^{-i}\right)^k&\leq\sum_{k=0}^\infty\mu(k)e^{-2kA(m)^{-i}} \notag
		\\
		&=\sum_{k=0}^{\lceil 2A(m)^{ {\rho}i}\rceil}\mu(k)\underbrace{e^{-2kA(m)^{-i}}}_{\leq 1}+\sum_{k=\lceil  2A(m)^{ {\rho}i}\rceil+1}^\infty\mu(k)\underbrace{e^{-2kA(m)^{-i}}}_{\leq e^{-4A(m)^{ {\rho}i}2A(m)^{-i}}} \notag
		\\
		&\leq\mu([0,2A(m)^{ {\rho}i}])+e^{-2A(m)^{ {\varkappa}i}}. \label{minstrel}
	\end{align}
	Next, let $n_0\in\LN$ such that
	$n_0\geq m_0$
	as well as 
	$\mu([0,2A(m)^{i}])\geq\frac{1}{2}$, where $m_0$ is given as in Lemma~\ref{E1_est}. Then we have
	\begin{align}
		\frac{\prod_{i=i_0}^m\left(\mu([0,2A(m)^{ {\rho}i}])+e^{-2A(m)^{ {\varkappa}i}}\right)}{\prod_{i=i_0}^m\mu([0,2A(m)^{ {\rho}i}])}&=\prod_{i=i_0}^m\left(1+\frac{e^{-2A(m)^{ {\varkappa}i}}}{\mu([0,2A(m)^{ {\rho}i}])}\right) \notag
		\\
		&\leq\underbrace{\prod_{i=i_0}^{n_0-1}\left(1+\frac{e^{-2A(m)^{ {\varkappa}i}}}{\mu([0,2A(m)^{ {\rho}i}])}\right)}_{=:C_1}\cdot\prod_{i=n_0}^m\left(1+2e^{-2A(m)^{ {\varkappa}i}}\right). \label{morose = sullen}
	\end{align}
	The first product is finite and the second one converges, since the series
	\begin{displaymath}
		\sum_{i=n_0}^\infty e^{-2A(n)^{ {\varkappa}i}}\leq\sum_{i=n_0}^\infty e^{-2A(n_0)^{ {\varkappa}i}}
	\end{displaymath}
	converges absolutely by the ratio test:
	\begin{displaymath}
		\frac{e^{-2A(n_0)^{ {\varkappa}(i+1)}}}{e^{-2A(n_0)^{ {\varkappa}i}}}=e^{2A(n_0)^{ {\varkappa}i}(1-2A(n_0)^{ {\varkappa}})}\leq e^{1-2A(n_0)^{ {\varkappa}}}<1.
	\end{displaymath}
	Therefore the fraction on the left hand side of 
	\eqref{morose = sullen} is bounded.
	This means that
	\begin{align*}
		\sum_{m=m_0}^\infty\prod_{i=i_0}^m\left(\mu([0,2A(m)^{ {\rho}i}])+e^{-2A(m)^{ {\varkappa}i}}\right)&=\sum_{m=m_0}^{n_0-1}\prod_{i=i_0}^m\left(\mu([0,2A(m)^{ {\rho}i}])+e^{-2A(m)^{ {\varkappa}i}}\right)
		\\
		&\hspace{20pt}+C_1\sum_{m=n_0}^\infty\prod_{i=n_0}^m\left(\mu([0,2A(m)^{ {\rho}i}])+e^{-2A(m)^{ {\varkappa}i}}\right)
		\\
		&\simeq\sum_{m=n_0}^\infty\prod_{i=n_0}^m\mu([0,2A(m)^{ {\rho}i}]).
	\end{align*}
	and therefore by \eqref{froth} and \eqref{minstrel}
	\begin{equation*}
		\sum_{m=m_0}^\infty\prod_{i=i_0}^m\sum_{k=0}^\infty\mu(k)(1-E_2(i,m))^i\precsim\sum_{m=m_0}^\infty\prod_{i=i_0}^m\mu([0,2A(m)^{{ {\rho}}i}]).
	\end{equation*}
\end{proof}
Finally, replacing
in Lemma \ref{shrew is a small animal}
the sequence $\{A(m)\}_{m\in\N}$ by $\{B(m)\}_{m\in\N}$ with $B(m)=2A(m)$ for all $m\in\N$ and putting all lemmas and calculations together, we arrive at the following result:
\begin{thm}\label{end_theorem}
	Assume that there exists a non-decreasing function $B\colon\N\to(0,\infty)$ such that
	\begin{displaymath}
		\sum_{z=1}^\infty\frac{1}{B(z)}<\infty.
	\end{displaymath}
	Furthermore, assume there exists $\rho>1$ such that the initial distribution of the frog model satisfies
	\begin{displaymath}
		\sum_{m=1}^\infty\prod_{i=1}^m\mu([0,B(m)^{\rho i}])<\infty.
	\end{displaymath}
	Then the frog process explodes a.s.
\end{thm}
\begin{proof}
	By the above calculation and Lemma \ref{L5.4}, we have that
	\begin{displaymath}
		\sum_{m=1}^\infty\prod_{i=0}^m(1-r_{m-i}^i(A))<\infty.
	\end{displaymath}
	By Lemma \ref{conv_perc}, the corresponding TADIBP percolates. By Proposition~\ref{perc_expl}, the system explodes.
\end{proof}

\bibliographystyle{alphaSinus}
\bibliography{Sinus}

\begin{thebibliography}{AMPR01}

\bibitem[ADGO17]{FFP_trees_expl}
O. Amini, L. Devroye, S. Griffiths, and N. Olver.
\newblock Explosion and linear transit times in infinite trees.
\newblock {\em Probab. Theory Relat. Fields}, 167(1-2):325--347, 2017.

\bibitem[AMP02]{shapeFrog}
O.~S.~M. {Alves}, F.~P. {Machado}, and S.~Y. {Popov}.
\newblock {The shape theorem for the frog model.}
\newblock {\em {Ann. Appl. Probab.}}, 12(2):533--546, 2002.

\bibitem[AMPR01]{shapeFrogRandom}
O.~S.~M. {Alves}, F.~P. {Machado}, S.~Y. {Popov}, and K. {Ravishankar}.
\newblock {The shape theorem for the frog model with random initial
  configuration.}
\newblock {\em {Markov Process. Relat. Fields}}, 7(4):525--539, 2001.

\bibitem[BDK21]{frog1}
V. {Bezborodov}, L. {Di Persio}, and T. {Krueger}.
\newblock {The continuous-time frog model can spread arbitrarily fast}.
\newblock {\em {Stat. Probab. Lett.}}, 172:7, 2021.
\newblock Id/No 109046.

\bibitem[{Bez}21]{TABP}
V. {Bezborodov}.
\newblock {Non-triviality in a totally asymmetric one-dimensional Boolean
  percolation model on a half-line}.
\newblock {\em {Stat. Probab. Lett.}}, 176:4, 2021.
\newblock Id/No 109155.

\bibitem[BFHM20]{BFHM20}
I. {Benjamini}, L.~R. {Fontes}, J. {Hermon}, and F.~P. {Machado}.
\newblock {On an epidemic model on finite graphs}.
\newblock {\em {Ann. Appl. Probab.}}, 30(1):208--258, 2020.

\bibitem[BK20]{frogL}
V. Bezborodov and T. Krueger.
\newblock Linear and superlinear spread for stochastic combustion growth
  process.
\newblock arXiv:2008.10585, 2020.
\newblock To appear in: \emph{Ann. Inst. Henri Poincaré Probab. Stat.}

\bibitem[BY16]{BY16}
J. {Bao} and C. {Yuan}.
\newblock {Blow-up for stochastic reaction-diffusion equations with jumps}.
\newblock {\em {J. Theor. Probab.}}, 29(2):617--631, 2016.

\bibitem[CD16]{CD16}
S. {Chatterjee} and P.~S. {Dey}.
\newblock {Multiple phase transitions in long-range first-passage percolation
  on square lattices}.
\newblock {\em {Commun. Pure Appl. Math.}}, 69(2):203--256, 2016.

\bibitem[CD21]{FrogsOnCayleyGraphs}
C.~F. {Coletti} and L.~R. {De Lima}.
\newblock {The asymptotic shape theorem for the frog model on finitely
  generated abelian groups}.
\newblock {\em {ESAIM, Probab. Stat.}}, 25:204--219, 2021.

\bibitem[CK14]{Chow13Explosion}
P.-L. {Chow} and R. {Khasminskii}.
\newblock {Almost sure explosion of solutions to stochastic differential
  equations}.
\newblock {\em {Stochastic Processes Appl.}}, 124(1):639--645, 2014.

\bibitem[CQR07]{CQR07}
F. {Comets}, J. {Quastel}, and A.~F. {Ram\'{\i}rez}.
\newblock {Fluctuations of the front in a stochastic combustion model.}
\newblock {\em {Ann. Inst. Henri Poincar\'e, Probab. Stat.}}, 43(2):147--162,
  2007.

\bibitem[CQR09]{CQR09}
F. {Comets}, J. {Quastel}, and A.~F. {Ram\'{\i}rez}.
\newblock {Fluctuations of the front in a one dimensional model of \(X+Y\to
  2X\).}
\newblock {\em {Trans. Am. Math. Soc.}}, 361(11):6165--6189, 2009.

\bibitem[CSKM13]{MeckeBook13}
S.~N. {Chiu}, D. {Stoyan}, W.~S. {Kendall}, and J. {Mecke}.
\newblock {\em {Stochastic geometry and its applications. 3rd revised and
  extended ed.}}
\newblock Hoboken, NJ: John Wiley \& Sons, 3rd revised and extended ed.
  edition, 2013.

\bibitem[DHL19]{DHL19}
M. {Deijfen}, T. {Hirscher}, and F. {Lopes}.
\newblock {Competing frogs on \({\mathbb Z}^d\).}
\newblock {\em {Electron. J. Probab.}}, 24:17, 2019.
\newblock Id/No 146.

\bibitem[EW03]{EW03}
A. {Eibeck} and W. {Wagner}.
\newblock {Stochastic interacting particle systems and nonlinear kinetic
  equations}.
\newblock {\em {Ann. Appl. Probab.}}, 13(3):845--889, 2003.

\bibitem[{Har}63]{HarrisBook}
T.~E. {Harris}.
\newblock {\em {The theory of branching processes}}, volume 119.
\newblock Springer, Cham, 1963.

\bibitem[{Her}18]{FrogsOnTrees?}
J. {Hermon}.
\newblock {Frogs on trees?}
\newblock {\em {Electron. J. Probab.}}, 23:40, 2018.
\newblock Id/No 17.

\bibitem[HJJ19a]{HJJ19}
C. {Hoffman}, T. {Johnson}, and M. {Junge}.
\newblock {Cover time for the frog model on trees}.
\newblock {\em {Forum Math. Sigma}}, 7:49, 2019.
\newblock Id/No e41.

\bibitem[HJJ19b]{FrogsOnTrees}
C. {Hoffman}, T. {Johnson}, and M. {Junge}.
\newblock {Infection spread for the frog model on trees}.
\newblock {\em {Electron. J. Probab.}}, 24:29, 2019.
\newblock Id/No 112.

\bibitem[IW89]{IkedaWat}
N. Ikeda and S. Watanabe.
\newblock {\em Stochastic differential equations and diffusion processes},
  volume~24 of {\em North-Holland Mathematical Library}.
\newblock North-Holland Publishing Co., Amsterdam; Kodansha, Ltd., Tokyo,
  second edition, 1989.

\bibitem[KW06]{KW06}
H.~G. {Kellerer} and G. {Winkler}.
\newblock {Random dynamical systems on ordered topological spaces.}
\newblock {\em {Stoch. Dyn.}}, 6(3):255--300, 2006.

\bibitem[{Lam}70]{Lamp70}
J. {Lamperti}.
\newblock {Maximal branching processes and 'long-range percolation'.}
\newblock {\em {J. Appl. Probab.}}, 7:89--98, 1970.

\bibitem[PP94]{PP94}
R. Pemantle and Y. Peres.
\newblock Domination between trees and application to an explosion problem.
\newblock {\em Ann. Probab.}, 22(1):180--194, 1994.

\bibitem[RS04]{stocCombust}
A.~F. {Ram\'{\i}rez} and V. {Sidoravicius}.
\newblock {Asymptotic behavior of a stochastic combustion growth process.}
\newblock {\em {J. Eur. Math. Soc. (JEMS)}}, 6(3):293--334, 2004.

\bibitem[vdHK17]{ExplFpp}
R. van~der Hofstad and J. Komjathy.
\newblock Explosion and distances in scale-free percolation.
\newblock 2017.
\newblock preprint; arXiv:1706.02597 [math.PR].

\bibitem[{Zer}18]{Zer18}
M.~P.~W. {Zerner}.
\newblock {Recurrence and transience of contractive autoregressive processes
  and related Markov chains.}
\newblock {\em {Electron. J. Probab.}}, 23:24, 2018.
\newblock Id/No 27.

\end{thebibliography}

\end{document}